\documentclass{amsart}
\usepackage{amsmath,amscd,xypic,amssymb,combelow,color,enumitem,graphicx,float,comment,mathtools}
\usepackage{tikz-cd}
\usepackage[hidelinks]{hyperref}
\definecolor{citation}{rgb}{0,.40,.80}
\hypersetup{
   colorlinks=true,
   citecolor=black,
    linkcolor=citation,
    }
\usepackage[style=alphabetic, backend=biber, eprint=true, maxbibnames=99, doi=false, url = false, isbn=false]{biblatex}
\usepackage[T2A]{fontenc}
\usepackage[russian,english]{babel}
\usepackage{quiver}
\usetikzlibrary{babel}
\newtheorem{theorem}{Theorem}[section]

\newtheorem{proposition}[theorem]{Proposition}
\newtheorem{lemma}[theorem]{Lemma}

\newtheorem{definition}[theorem]{Definition}
\newtheorem{claim}[theorem]{Claim}

\newtheorem{remark}[theorem]{Remark}

\def\GL{\mathrm{GL}}

\def\BA{{\mathbb{A}}}
\def\BC{{\mathbb{C}}}

\def\BN{{\mathbb{N}}}
\def\BF{{\mathbb{F}}}

\def\BZ{{\mathbb{Z}}}

\DeclareMathOperator{\Tr}{\mathrm{Tr}}

\DeclareMathOperator{\Db}{\mathrm{D}^{\mathrm{b}}}

\def\CA{{\mathcal{A}}}

\def\DD{{\mathcal{D}}}

\def\CS{{\mathcal{S}}}

\def\CV{{\mathcal{V}}}

\def\ph{\varphi}

\def\sym{\textrm{sym}}

\def\ba{{\mathbf{a}}}
\def\bb{{\mathbf{b}}}

\def\bd{{\mathbf{d}}}

\def\bk{{\mathbf{k}}}

\def\br{{\mathbf{r}}}

\def\bs{\boldsymbol{\varsigma}}
\def\bell{\boldsymbol{\ell}}

\def\nn{{\mathbb{N}^I}}
\def\zz{{\mathbb{Z}^I}}

\newcommand{\blambda}{\boldsymbol{\lambda}}

\def\bm{{\boldsymbol{m}}}
\def\bn{{\boldsymbol{n}}}

\def\b0{{\boldsymbol{0}}}

\def\loccit{\emph{loc.~cit.~}}

\def\hdeg{\text{hdeg }}
\def\vdeg{\text{vdeg }}

\def\edge{\Omega}
\def\dedge{\bar{\edge}}
\def\dalpha{\bar{\alpha}}
\def\dQ{\bar{Q}}
\def\tQ{\tilde{Q}}

\def\tW{\tilde{W}}

\def\ring{R}
\def\field{\BF}

\def\ozeta{\bar{\zeta}}
\def\short{\prod_{(\alpha : i \rightarrow j) \in \edge\text{ or } (\dalpha:j \rightarrow i)\in \edge}}

\def\soft{\begin{otherlanguage*}{russian}Ь \end{otherlanguage*}}
\def\hard{\begin{otherlanguage*}{russian}Ъ \end{otherlanguage*}}

\def\Ts{\Phi}
\def\eTs{\Phi}

\def\kha{K^{T,\omega\text{-nilp}}_{\tQ,\tW}}
\def\kzero{K^{T}_{\tQ}}

\def\khan{K^{T,\omega\text{-nilp}}_{\tQ,\tW,\bn}}

\def\kzeroloc{K^{T}_{\tQ,\text{loc}}}

\def\ekha{K^{T,\omega\emph{-nilp}}_{\tQ,\tW}}

\def\ekhan{K^{T,\omega\emph{-nilp}}_{\tQ,\tW,\bn}}

\DeclareMathOperator{\Perf}{\mathrm{Perf}}
\DeclareMathOperator{\fX}{\mathfrak{X}}
\DeclareMathOperator{\Coh}{\mathrm{Coh}}

\def\point{(\cdot)}

\def\nilpstack{\mathfrak{Y}}

\begin{document}

\title[$K$-theoretic Hall Algebras and Coulomb branches]{\Large{\textbf{$K$-theoretic Hall algebras and Coulomb branches}}} 

\author[Shivang Jindal and Andrei Negu\cb t]{Shivang Jindal and Andrei Negu\cb t}

\address{École Polytechnique Fédérale de Lausanne (EPFL), Lausanne, Switzerland} 
\email{shivang.jindal@epfl.ch}

\address{École Polytechnique Fédérale de Lausanne (EPFL), Lausanne, Switzerland \newline \text{ } \ \ Simion Stoilow Institute of Mathematics (IMAR), Bucharest, Romania} 
\email{andrei.negut@gmail.com}
	
\begin{abstract} We construct a surjective homomorphism from the (suitably interpreted) double loop-nilpotent $K$-theoretic Hall algebra to the Coulomb branch algebra of a quiver gauge theory, using the shuffle algebra interpretation.
\end{abstract}

\maketitle

\bigskip

\section{Introduction}
\label{sec:intro}

\subsection{Branches of quiver gauge theories}

Consider a quiver $Q$ whose arrows are endowed with equivariant parameters as in Subsection \ref{sub:quivers}. To this data, one can associate a quiver gauge theory, which comes with two algebraic varieties called the Higgs and Coulomb branches respectively. These two varieties are expected to be closely related to each other via a phenomenon called 3-dimensional mirror symmetry, made mathematically precise by the Hikita-Nakajima conjecture \cite{Hikita} (see also \cite{BLPW, KMP}, also \cite{Zhou} in $K$-theory, and numerous other works).

\medskip 

\noindent In geometric representation theory, one is interested in the algebras associated to the two branches, as follows:

\medskip 

\begin{itemize}[leftmargin=*]

\item the ``Higgs branch algebra'' refers to a certain incarnation of the $K$-theoretic Hall algebra of $Q$, see \cite{KS, Padurariu, SVHilb}. 

\medskip 

\item the ``Coulomb branch algebra'' refers to a certain deformed algebra of functions on the space defined in \cite{Nakajima, BFN}.
    
\end{itemize}

\medskip 

\noindent Let $K^T_{\Pi_Q}$ be the torus $T$ equivariant $K$-theoretic preprojective Hall algebra of $Q$, that was introduced in \cite{SVHilb, Yang_2018}. It is the natural algebraic object in the first bullet above, as it acts on the $K$-theory groups of all Higgs branches. By a well-known phenomenon known as dimensional reduction, there is an isomorphism
\begin{equation}
\label{eqn:k-ha intro}
K^T_{\Pi_Q} \cong K_{\tQ,\tW}^T
\end{equation}
where the right-hand side is the $K$-theoretic Hall algebra (abbreviated $K$-HA, and introduced in \cite{KS, Padurariu}) of the tripled quiver with potential $(\tQ,\tW)$ that we recall in Subsection \ref{sub:k-ha}. Because of the isomorphism above, we will mostly be concerned with the algebra in the right-hand side of \eqref{eqn:k-ha intro} throughout the present paper. In fact,  our precise object of study is an $\ring = K^T\point$ algebra denoted by
\begin{equation}
\label{eqn:higgs intro}
\kha
\end{equation}
which geometrically corresponds to the loop-nilpotent version of the aforementioned $K$-HA. Using the shuffle algebra interpretation of \eqref{eqn:higgs intro}, we show how to construct a shifted double of \eqref{eqn:higgs intro}. Meanwhile, the $K$-theoretic Coulomb branch is an $\ring$-algebra
\begin{equation}
\label{eqn:coulomb intro}
\CA_{\bd|\bk,\bell}
\end{equation}
defined in \cite{BFN, BFNQuiver} for any $\bd \in \nn$ and framing $I$-tuples $\bk,\bell \in \nn$ as in \eqref{eqn:framing}. The main purpose of the present paper is to connect the Higgs and Coulomb branch algebras, by constructing an explicit algebra homomorphism between them.

\medskip 

\begin{theorem}
\label{thm:main}

For any $\bd,\bk,\bell \in \nn$, there is a surjective $\ring$-algebra homomorphism
\begin{equation}
\label{eqn:main}
\Big(\bd^\vee-\bk-\bell\text{ shifted Drinfeld double of }\ekha \Big) \twoheadrightarrow \CA_{\bd|\bk,\bell} 
\end{equation}
under the Assumption \soft of \eqref{eqn:assumption soft}, where $\bd^\vee$ is defined in \eqref{eqn:vee}.
    
\end{theorem}

\medskip 

\noindent While we do not have a geometric definition of the shifted Drinfeld double of $K$ that appears in the left-hand side of \eqref{eqn:main}, we will use its shuffle algebra interpretation to define a suitable incarnation of this double. Specifically, the rigorous object that we place in the left-hand side of \eqref{eqn:main} is the algebra \eqref{eqn:rigorous main}, to which we now turn.

\medskip 

\subsection{Shuffle algebras for the Higgs branch}
\label{sub:shuffle higgs}

As shown in \cite{KS, Padurariu}, the $K$-theoretic Hall algebra associated to a quiver with potential admits the following shuffle algebra interpretation. Consider
$$
\CV^+ = \bigoplus_{\bn = (n_i \geq 0)_{i \in I} \in \nn} R[z_{i1}^{\pm 1},\dots,z_{in_i}^{\pm 1}]_{i \in I}^{\sym}
$$
with multiplication given by \eqref{eqn:mult}. From the foundational works \cite{KS, Padurariu}, it is well-known that there is an $\ring$-algebra homomorphism
$$
\iota : \kha \rightarrow \CV^+
$$
Under the geometric assumption \eqref{eqn:assumption geometric}, we prove $\iota$ to be injective in Proposition \ref{thm:injective}. Under Assumption \soft of \eqref{eqn:assumption soft}, we identify in Proposition \ref{prop:generator} the image of $\iota$ with the explicit shuffle algebra $\CS^+ \subset \CV^+$ that we introduce in Definition \ref{def:shuffle int}. Combining these two statements gives us an isomorphism
\begin{equation}
\label{eqn:k shuffle intro}
\kha \xrightarrow{\sim} \CS^+
\end{equation}
Using the topological bialgebra structure on shuffle algebras that we recall in Subsections \ref{sub:coproduct} and \ref{sub:drinfeld double}, one may define a shifted double shuffle algebra
\begin{equation}
\label{eqn:double shuffle intro}
\CS^{\ba-\bb}_{\text{loc}} =  \CS^+_{\text{loc}} \otimes (\text{Cartan subalgebra}) \otimes \CS^{+,\text{op}}_{\text{loc}}
\end{equation}
where ``loc'' refers to tensoring all algebras with $\field = \text{Frac}(\ring)$, and the superscript $\ba-\bb$ refers to the shift with respect to arbitrary $\ba,\bb \in \zz$ (see Section \ref{sec:double}, and praticularly equation \eqref{eqn:dd shifted}, for details). In light of the map \eqref{eqn:k shuffle intro}, we define
\begin{equation}
\label{eqn:rigorous main}
\CS^{\ba-\bb} = \Big(\ba - \bb\text{ shifted Drinfeld double of }\CS^+ \Big)
\end{equation}
as the $\ring$-subalgebra of $\CS^{\ba-\bb}_{\text{loc}}$ generated by $\CS^+ \subset \CS^+_{\text{loc}}$, $\CS^{+,\text{op}} \subset \CS^{+,\text{op}}_{\text{loc}}$, and a certain integral form of the Cartan subalgebra that we review in Subsection \ref{sub:integral double}. We will explain at the end of Section \ref{sec:double} that the presentation above provides tools to explicitly compute the commutation relations of arbitrary elements from the two halves of the shifted double of $\kha$.

\medskip 

\subsection{Shuffle algebras for the Coulomb branch}
\label{sub:shuffle coulomb}

Meanwhile, it was recognized by Finkelberg-Tsymbaliuk in \cite{FT}, Frassek-Tsymbaliuk in \cite{FrT} and Tsymbaliuk in \cite{Tsymbaliuk} that shuffle algebras also provide models for the Coulomb branch algebra. Specifically, it has been well-known since the work of \cite{BFN, BFNQuiver} that for any $\bd \in \nn$ and framing parameters $\bk,\bell$, there is an injective algebra homomorphism
\begin{equation}
\label{eqn:algebra intro}
\CA_{\bd|\bk,\bell} \hookrightarrow \DD_{\bd|\bk,\bell}
\end{equation}
where the codomain refers to the ring of $q$-difference operators \eqref{eqn:difference operators}. Following ideas of \cite{GKLO}, Finkelberg, Frassek and Tsymbaliuk defined a homomorphism
\begin{equation}
\label{eqn:tsymbaliuk intro}
\CS^{\bd^\vee - \bk - \bell}_{\text{loc}} \rightarrow \DD_{\bd|\bk,\bell}
\end{equation}
where the shift parameters $\bd^\vee - \bk - \bell$ are recalled in Subsection \ref{sub:gklo}. In the present paper, we restrict the homomorphism \eqref{eqn:tsymbaliuk intro} to the integral form \eqref{eqn:rigorous main} and obtain an $\ring$-algebra homomorphism 
\begin{equation}
\label{eqn:morphism intro}
\CS^{\bd^\vee - \bk - \bell} \rightarrow \DD_{\bd|\bk,\bell}
\end{equation} 
Our key observation for the proof of Theorem \ref{thm:main} is that the image of \eqref{eqn:morphism intro} precisely coincides with the image of \eqref{eqn:algebra intro}, which was described in \cite{BFN, DK, FT, SS, Weekes}. This implies that we have a surjective $\ring$-algebra homomorphism 
\begin{equation}
\label{eqn:final morphism intro}
\CS^{\bd^\vee - \bk - \bell} \twoheadrightarrow \CA_{\bd|\bk,\bell}
\end{equation} 
Combining the surjection above with the definition \eqref{eqn:rigorous main} yields the surjection \eqref{eqn:main}. 

\medskip 

\subsection{Relation to previous work} 
\label{sub:relation}

The existence of a homomorphism \eqref{eqn:main} has been suggested by numerous mathematicians and physicists, see for instance \cite{RSYZ} and \cite{CL}, as well as the seminal paper \cite{Nakajima}. In the present paper, we make this homomorphism precise by

\medskip 

\begin{itemize}[leftmargin=*]

\item interpreting the specific $K$-HA which appears in the homomorphism \eqref{eqn:main} as the loop-nilpotent version of this construction;

\medskip 

\item identifying the loop-nilpotent $K$-HA with a specific shuffle algebra by new wheel conditions (defined over $\ring$ instead of over $\field = \text{Frac}(\ring)$);

\medskip 

\item finding generators for the aforementioned shuffle algebra over the ring $\ring$. While the shuffle algebra is spherically generated over $\field$ (\cite{Wheel}) it is not so over $\ring$;

\medskip 

\item using the shuffle algebra incarnation to define the shifted double loop-nilpotent $K$-HA in \eqref{eqn:rigorous main}. It would be very interesting to define a double $K$-HA geometrically.

\end{itemize}

\medskip 

\noindent The loop-nilpotent CoHA and K-HA have been studied previously, see for instance \cite{davison2022integrality, DHKSV, VV, J}. Meanwhile, integral forms of shuffle algebras have been studied by \cite{FT, TsymbaliukPBW} in type $A$ and by the second-named author in \cite{Integral} for the Jordan quiver. In fact, the particular integral form of the shuffle algebra we consider in the present paper is the natural generalization of the construction of \loccit to arbitrary quivers. The presentation of Coulomb branch algebras by generators and difference operators, which plays a key role in the present paper, has been studied in \cite{BFN, DK, FT, KWWY, Weekes} and many other works.

\medskip 

\noindent We would also like to mention the upcoming paper \cite{MW}, developed independently from ours. Among other things, the authors of \loccit construct a injective homomorphism from a rational shuffle algebra to a suitably defined limit of (the cohomological version of) the Coulomb branch algebras $\CA_{\bd|\bk,\bell}$ as $\bd \rightarrow \infty$. We thank Dinakar Muthiah for informing us about the details of their project.

\subsection{Acknowledgements} We would like to thank Michael Finkelberg, Dinakar Muthiah, Tudor Pădurariu, Olivier Schiffmann, Alexander Shapiro, Alexander Tsymbaliuk and Alex Weekes for numerous interesting conversations. We gratefully acknowledge the support of the Swiss National Science Foundation grant 10005316.

\bigskip 

\section{The Higgs branch}
\label{sec:higgs}

\medskip

\subsection{Quivers}
\label{sub:quivers}

We follow the conventions of \cite{SS} and \cite{Tsymbaliuk}. Fix a quiver $Q$ with vertex set $I$ and arrow set $\edge$. Consider a field $\field$ of characteristic 0 endowed with elements
\begin{equation}
\label{eqn:parameters}
\left\{ t_{\alpha}, q \right\}_{\alpha \in \edge} \subset \field^*
\end{equation}
and define the subring 
$$
\ring \subset \field \quad \text{generated by }\left\{t_{\alpha}^{\pm 1},q^{\pm 1} \right\}_{\alpha \in \edge}
$$ 
Consider the doubled quiver $\dQ$ with vertex set $I$, whose arrow set $\dedge$ consists of
\begin{equation}
\label{eqn:double parameters}
i \xleftrightharpoons[t_{\alpha}]{q/t_{\alpha}} j
\end{equation}
for every arrow $i \rightarrow j$ in $\edge$. The symbols $t_\alpha, q/t_\alpha$ above represent parameters in $\ring$ associated to the arrows in the double quiver. Alternatively, if we write $\dalpha : j \rightarrow i$ for the opposite arrow to any $\alpha : i \rightarrow j$ in $\dedge$, then the parameter of $\dalpha$ is defined as
\begin{equation}
\label{eqn:opposite parameters}
t_{\dalpha} = \frac q{t_{\alpha}}
\end{equation}
In the present paper, we will encounter the following genericity assumptions on the parameters \eqref{eqn:parameters}  \footnote{Assumption \hard originated in \cite{Wheel}, where it appeared in the slightly stronger form
\begin{equation}
\label{eqn:assumption}
\exists \text{ a field homomorphism } \rho : \field \rightarrow \BC \text{ s.t. } |\rho(q)| < |\rho(t_\alpha)| < 1, \ \forall \alpha \in \edge
\end{equation}
However, we know of no important situation where \eqref{eqn:assumption hard} holds but \eqref{eqn:assumption} fails.}
\begin{equation}
\label{eqn:assumption hard}
\text{Assumption \hard}: \quad q \text{ is not a root of unity, and} 
\end{equation}
$$
\text{the equation } \prod_{\alpha \in \edge} t_{\alpha}^{x_{\alpha}} t_{\dalpha}^{y_{\alpha}} =1 \text{ has no non-trivial solution }x_\alpha, y_\alpha \in \BZ_{\geq 0}
$$
\begin{equation}
\label{eqn:assumption soft}
\text{Assumption \soft}: \quad q \text{ is not a root of unity, and} 
\end{equation}

\begin{itemize}[leftmargin=*]

\item $\left\{t_{\alpha}q^{\BZ_{< 0}}, t_{\beta}^{-1}q^{\BZ_{> 0}}\right\}_{i \xleftrightharpoons[\alpha]{\beta} j} \cap \left\{t_{\alpha}q^{\BZ_{\geq 0}}, t_{\beta}^{-1}q^{\BZ_{\leq 0}} \right\}_{i \xleftrightharpoons[\alpha]{\beta} j} = \varnothing, \ \forall  i,j \in I, \alpha,\beta \in \edge$

\medskip 

\item $\left(t_\alpha q^{\BZ_{<0}} \right)^x = \left(t_\beta q^{\BZ_{\geq 0}} \right)^y \text{ for }x,y 
\in \BZ \text{ implies }x = y = 0, \forall \alpha,\beta \in \edge_{\text{loop}}$

\end{itemize}

\medskip

\noindent (above, $\edge_{\text{loop}} \subset \edge$ denotes the set of loops of $Q$) as well as the 
\begin{equation}
\label{eqn:assumption geometric}
\text{Geometric assumption}: \quad q \text{ is not a root of unity and } t_{\alpha} \notin q^{\BZ}, \forall \alpha \in \Omega_{\text{loop}}
\end{equation}

\medskip 

\subsection{Conventions}
\label{sub:conventions}

In the present paper, the set $\BN$ will contain 0. We write $\b0 = (0,\dots,0) \in \nn$ and let $\bs^i \in \nn$ be the $I$-tuple with a single 1 at position $i$ and 0 everywhere else. The adjacency matrix of the quiver $Q$ induces the bilinear form
$$
\zz \times \zz \xrightarrow{\langle \cdot, \cdot \rangle} \BZ
$$
determined by
$$
\langle \bs^i, \bs^j \rangle = \delta_{ij} - \Big|\text{arrows }i \rightarrow j \Big|
$$
The Euler form refers to the symmetrization of the above bilinear form
$$
\zz \times \zz \xrightarrow{( \cdot, \cdot )} \BZ
$$
determined by
$$
( \bs^i, \bs^j ) = 2\delta_{ij} - \Big|\text{arrows }i \rightarrow j\Big| - \Big|\text{arrows }j \rightarrow i\Big|
$$
Given $\bd \in \zz$, we will write $\bd^\vee$ for the $I$-tuple of integers whose entries are
\begin{equation}
\label{eqn:vee}
d^\vee_i = 2d_i - \sum_{j \in I} d_j \left( \Big|\text{arrows }i \rightarrow j\Big| + \Big|\text{arrows }j \rightarrow i\Big| \right)
\end{equation}
so that for any $\bn \in \zz$ the usual dot product of $I$-tuples satisfies the equation
$$
\bd^\vee \cdot \bn = (\bd,\bn)
$$
For any $\bm = (m_i)_{i \in I}, \bn = (n_i)_{i \in I}$, we write $\bm \leq \bn$ if $m_i \leq n_i$ for all $i \in I$. Let
$$
|\bn| = \sum_{i \in I} n_i
$$

\medskip 

\subsection{$K$-theoretic Hall algebras}
\label{sub:k-ha}

Just as we obtained $\dQ$ from $Q$ by doubling all the arrows, we construct the tripled quiver $\tQ$ from $\dQ$ by adding a loop $\omega_i$ at every vertex $i \in I$. We consider the canonical cubic potential
\begin{equation}
\label{eqn:potential}
\tW = \sum_{\alpha : i \rightarrow j} \Big( \dalpha \alpha \omega_i - \omega_j \alpha \dalpha \Big)
\end{equation}
Let us consider a torus $T = (\BC^*)^r$ endowed with characters 
$$
\Big\{t_\alpha,q : T \rightarrow \BC^*\Big\}_{\alpha \in \edge}
$$
such that $T$ acts on the arrows $\alpha,\dalpha,\omega_i$ via the characters $t^{-1}_\alpha, t^{-1}_{\dalpha},q$, respectively (see Subsection \ref{sub:moduli} for details). This choice ensures that the potential $\tW$ is $T$-invariant. In the language of Subsection \ref{sub:quivers}, denote the $K$-theory of a point by
\begin{equation}
\label{eqn:ring and field}
\ring = K^T\point = \text{Rep}_T
\end{equation}
and $\field = \text{Frac}(\ring)$. To the quiver with potential $(\tQ, \tW)$, one associates a $K$-theoretic Hall algebra ($K$-HA for short) as in \cite{KS, Padurariu}. In fact, in the present paper we need to consider the loop-nilpotent version of this algebra, which we denote by
$$
\kha 
$$
and review in detail in Appendix B. There is an $\ring$-algebra homomorphism 
\begin{equation}
\label{eqn:k-ha to shuffle}
\iota : \kha \rightarrow \kzero 
\end{equation}
defined as in \cite{Padurariu}, where the codomain refers to the $K$-theoretic Hall algebra of the tripled quiver with 0 potential. The following results will be proved in Appendix B, Section \ref{sec:b}.

\medskip 

\begin{theorem}
\label{thm:injective}

Under the geometric assumption \eqref{eqn:assumption geometric}, the map \eqref{eqn:k-ha to shuffle} is injective.

\end{theorem}

\medskip 

\noindent For any $\bn \in \nn$ and $g \in K^{T \times \prod_{i \in I} \GL_{n_i}(\BC)}\point$, we will define elements
$$
\varepsilon_{\bn,g} \in \kha
$$
in Subsection \ref{sub:kha generators}, that correspond under dimensional reduction to tautological classes on the locus where the loops are all zero: $\{\omega_i = 0, \forall i \in I\}$.

\medskip 

\begin{proposition}
\label{prop:generator}

Under Assumption \soft of \eqref{eqn:assumption soft}, the classes
$$
\Big\{\iota(\varepsilon_{\bn,g}) \Big\}_{\bn \in \nn, g \in K^{T \times \prod_{i \in I} \GL_{n_i}(\BC)}\point}
$$
generate $\emph{Im } \iota$.

\end{proposition}

\medskip 

\noindent We hope that Proposition \ref{prop:generator} holds under the weaker geometric assumption \eqref{eqn:assumption geometric}.

\medskip 

\subsection{Big shuffle algebras}
\label{sub:big shuffle}

Corresponding to a quiver with parameters as in Subsection \ref{sub:quivers}, we associate the rational function (the normalization below matches that of \cite{Wheel}, up to certain powers of $t_\alpha$ that will be unimportant for us) 
\begin{equation}
\label{eqn:zeta}
\zeta_{ij}(x) = \left( \frac {1-xq^{-1}}{1-x} \right)^{\delta_{ij}} \prod_{(\alpha : i \rightarrow j) \in \edge} \left(1 - x t_\alpha \right) \prod_{(\alpha : j \rightarrow i) \in \edge} \left(1 - \frac {t_{\alpha}}{xq}\right) \in \ring(x)
\end{equation}
This choice allows us to place associative algebra structures on
\begin{align*}
&\CV^+ = \bigoplus_{\bn = (n_i \geq 0)_{i \in I} \in \nn} R[z_{i1}^{\pm 1},\dots,z_{in_i}^{\pm 1}]_{i \in I}^{\sym} \\
&\CV^+_{\text{loc}} = \bigoplus_{\bn = (n_i \geq 0)_{i \in I} \in \nn} \BF[z_{i1}^{\pm 1},\dots,z_{in_i}^{\pm 1}]_{i \in I}^{\sym}
\end{align*}
(above, ``sym'' refers to polynomials which are color-symmetric, i.e. symmetric in the variables $z_{i1},\dots,z_{in_i}$ for each $i \in I$ separately). Explicitly, the multiplication on the sets above is given by
\begin{equation}
\label{eqn:mult}
E( z_{i1}, \dots, z_{i n_i})_{i \in I} * E'(z_{i1}, \dots,z_{i n'_i})_{i \in I} = 
\end{equation}
$$
\textrm{Sym} \left[ \frac {E(z_{i1}, \dots, z_{in_i}) E'(z_{i,n_i+1}, \dots, z_{i,n_i+n'_i})}{\bn! \bn'!}
\prod_{i,j \in I} \mathop{\prod_{1 \leq a \leq n_i}}_{n_j < b \leq n_j+n_j'} \zeta_{ij} \left( \frac {z_{ia}}{z_{jb}} \right) \right]
$$
with ``Sym'' in \eqref{eqn:mult} referring to summing over the
\begin{equation*}
(\bn+\bn')! := \prod_{i\in I} (n_i+n'_i)!
\end{equation*}
permutations of the variables $\{z_{i1}, \dots, z_{i,n_i+n'_i}\}$ for each $i$ independently. We call $\CV^+$ the \emph{big shuffle algebra} and $\CV^+_{\text{loc}}$ its localized version. They are related to the constructions in Subsection \ref{sub:k-ha} by the existence of natural algebra isomorphisms 
\begin{align}
&\kzero \cong \CV^+ \label{eqn:iota zero} \\
&\kzeroloc \cong \CV^+_{\text{loc}} \label{eqn:iota zero loc}
\end{align}
(see \cite{KS, Padurariu}) where the subscript ``loc'' refers to tensoring with $\field$ over $\ring$. Formula \eqref{eqn:k-ha to shuffle} then yields algebra homomorphisms
\begin{align}
&\kha \rightarrow \CV^+ \label{eqn:iota} \\
&K_{\tQ,\tW,\text{loc}}^T \rightarrow \CV^+_{\text{loc}} \label{eqn:iota loc}
\end{align}
Our Theorem \ref{thm:injective} proves \eqref{eqn:iota} to be injective under the geometric assumption \eqref{eqn:assumption geometric}, while \cite{VV} proved that \eqref{eqn:iota loc} is injective under Assumption \eqref{eqn:assumption hard} (see \cite{Wheel}). 
 
\medskip

\subsection{Small shuffle algebras}
\label{sub:small shuffle}

We now introduce certain subalgebras of $\CV^+$, $\CV^+_{\text{loc}}$.

\medskip 

\begin{definition} 
\label{def:shuffle loc}

Let $\CS^+_{\emph{loc}} \subset \CV^+_{\emph{loc}}$ denote the set of $E(z_{i1},\dots,z_{in_i})_{i \in I}$ such that
$$
E\Big|_{z_{i2} = qz_{i1}} \quad \text{ is divisible by } \short (z_{jb} - z_{i1} t_\alpha)
$$
for all $i,j \in I$ and any $b$ (where we require $b>2$ if $i=j$).

\end{definition} 

\medskip 

\noindent The following was the main result of \cite{Wheel}.

\medskip 

\begin{theorem}
\label{thm:shuffle loc}

Under Assumption \hard of \eqref{eqn:assumption hard}, the set $\CS^+_{\emph{loc}}$ coincides with the
\begin{equation}
\label{eqn:spherical}
\field\text{-algebra generated by }\left\{ e_{i,k} = z_{i1}^k \right\}_{i \in I, k \in \BZ} \subset \CV^+
\end{equation}
where $z_{i1}^k$ is interpreted as a Laurent polynomial in a single variable of color $i$.

\end{theorem}

\medskip 

\noindent As for the integral shuffle algebra, we propose the following generalization of the construction of \cite[Definition 3.2]{Integral} (we also note the related type A constructions of \cite{FT, TsymbaliukPBW}, though we have not proved that they are equivalent to ours). We will refer to
\begin{equation}
\label{eqn:composition}
P = \left\{ n_i = n_i^{(1)}+\dots+n_i^{(d_i)} \right\}_{i \in I}
\end{equation}
as an $I$-composition of $\bn = (n_i)_{i \in I}$. We further call $P$ an $I$-partition if we have
$$
n_i^{(1)} \geq \dots \geq n_i^{(d_i)} > 0
$$
for all $i \in I$.

\medskip 

\begin{definition} 
\label{def:shuffle int}

Let $\CS^+ \subset \CV^+$ denote the set of Laurent polynomials
$$
E(z_{i1},\dots,z_{in_i})_{i \in I} \in \ring[z_{i1}^{\pm 1},\dots,z_{in_i}^{\pm 1}]^{\emph{sym}}_{i \in I}
$$
such that for any $I$-composition $P$ as in \eqref{eqn:composition}, the specialization
\begin{equation}
\label{eqn:specialization}
\emph{Spec}_P(E) = E \left(x_{i1},x_{i1}q,\dots,x_{i1}q^{n_i^{(1)}-1}, \dots, x_{id_i},x_{id_i}q,\dots,x_{id_i} q^{n_i^{(d_i)}-1} \right)_{i \in I}
\end{equation}
is divisible in the ring $\ring[z_{i1}^{\pm 1},\dots,z_{in_i}^{\pm 1}]^{\emph{sym}}_{i \in I}$ by
\begin{multline}
\label{eqn:factor}
\prod_{i \in I} \prod_{a=1}^{d_i} \left[(1-q^{-1})(1-q^{-2})\dots(1-q^{-n_i^{(a)}}) \right] \\ \prod_{(\alpha : i \rightarrow j) \in \edge} \mathop{\prod_{1 \leq a \leq d_i}}_{1 \leq b \leq d_j} \prod_{c \in \BZ} \left(1- \frac {x_{ia}t_{\alpha}}{x_{jb}q^{c}} \right)^{\chi_{n_i^{(a)},n_j^{(b)}}(c)}
\end{multline}
where for any $n,n' \geq 1$, let $\chi_{n,n'} : \BZ \rightarrow \BN$ be the function with the following graph

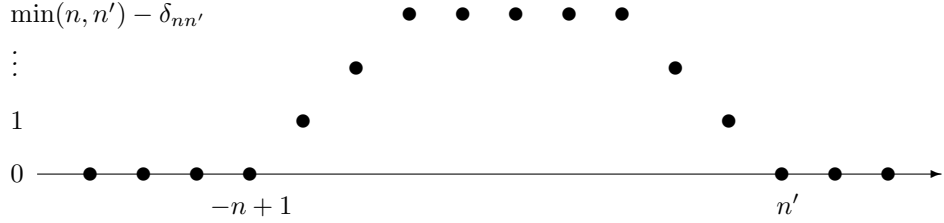
\begin{picture}(400,120)(0,-30)
\label{pic:par}

\put(20,0){\circle*{5}}\put(40,0){\circle*{5}}\put(60,0){\circle*{5}}\put(80,0){\circle*{5}}\put(100,20){\circle*{5}}\put(120,40){\circle*{5}}\put(140,60){\circle*{5}}\put(160,60){\circle*{5}}\put(180,60){\circle*{5}}\put(200,60){\circle*{5}}\put(220,60){\circle*{5}}\put(240,40){\circle*{5}}\put(260,20){\circle*{5}}\put(280,0){\circle*{5}}\put(300,0){\circle*{5}}\put(320,0){\circle*{5}}

\put(0,0){\vector(1,0){340}}
\put(65,-15){$-n+1$}
\put(278,-15){$n'$}

\put(-10,-3){$0$}\put(-10,17){$1$}\put(-10,37){$\vdots$}\put(-10,57){$\min(n,n')-\delta_{nn'}$}

\end{picture}

\end{definition} 

\medskip 

\noindent We call $\CS^+$ thus defined the \emph{integral shuffle algebra}. Note that due to the color-symmetry of $E$, Definition \ref{def:shuffle int} would not change if we restricted attention only to $I$-partitions $P$. The following generalization of Proposition 3.9 of \cite{Integral} will be proved in Appendix A, Section \ref{sec:a}.

\medskip 

\begin{theorem}
\label{thm:shuffle int}

Under the geometric assumption \eqref{eqn:assumption geometric}, the set $\CS^+$ coincides with the $\ring$-subalgebra of $\CV^+$ generated by
\begin{equation}
\label{eqn:special}
e_{\bn,g} = g(z_{i1},\dots,z_{in_i})_{i \in I} \prod_{i \in I} \prod_{1 \leq a , b \leq n_i} \left(1 - \frac {z_{ia}}{z_{ib}q}\right)
\end{equation}
as $\bn = (n_i)_{i \in I}$ ranges over $\nn$ and $g$ ranges over $\ring[z_{i1}^{\pm 1},\dots,z_{in_i}^{\pm 1}]^{\emph{sym}}_{i \in I}$.

\end{theorem}

\medskip 

\noindent It was shown in \cite{Wheel} that, under \text{Assumption \hard} of \eqref{eqn:assumption hard}, the map \eqref{eqn:iota loc} induces an isomorphism
\begin{equation}
\label{eqn:k shuffle loc}
K_{\tQ,\tW,\text{loc}}^T \xrightarrow{\sim} \CS^+_{\text{loc}}
\end{equation}
The following analogous result follows from Theorem \ref{thm:injective}, Proposition \ref{prop:generator}, Theorem \ref{thm:shuffle int}, together with the observation that $\iota(\varepsilon_{\bn,g}) = e_{\bn,g}$ for all $\bn,g$.

\medskip 

\begin{theorem}
\label{thm:geometry}

Under Assumption \soft of \eqref{eqn:assumption soft}, the map \eqref{eqn:iota} induces an isomorphism
\begin{equation}
\label{eqn:k shuffle int}
\ekha\xrightarrow{\sim} \CS^+
\end{equation}
We conjecture that \eqref{eqn:k shuffle int} holds even in the absence of Assumption \soft.

\end{theorem}

\bigskip 

\section{Doubles}
\label{sec:double}

\medskip 

\subsection{The negative shuffle algebras}

The (big) shuffle algebras with superscript $-$ will refer to the opposites of the ones introduced in the previous Section
\begin{align*}
&\CV^- = \CV^{+,\text{op}}, \qquad \CV^-_{\text{loc}} = \CV^{+,\text{op}}_{\text{loc}} \\
&\CS^- = \CS^{+,\text{op}}, \qquad \CS^-_{\text{loc}} = \CS^{+,\text{op}}_{\text{loc}}
\end{align*}
We will denote the analogous elements to \eqref{eqn:spherical} as $f_{i,k} \in \CV^-_{\text{loc}}$ and to \eqref{eqn:special} as
\begin{equation}
\label{eqn:special transpose}
f_{\bn,g} \in \CV^-
\end{equation}
They are given by the exact same formulas as their $e$ counterparts, but are viewed as elements in the opposite shuffle algebras $\CV^-$ and $\CV^-_{\text{loc}}$. We define the  horizontal and vertical gradings on shuffle algebras (valued in $\zz$ and $\BZ$, respectively) as
\begin{align}
&\hdeg E = \bn, \quad \ \ \vdeg E = k \label{eqn:deg 1} \\
&\hdeg F = -\bn, \quad \vdeg F = k \label{eqn:deg 2}
\end{align}
for any $E(z_{i1},\dots,z_{in_i})_{i \in I} \in \CV^+_{\text{loc}}$ and $F(z_{i1},\dots,z_{in_i})_{i \in I} \in \CV^-_{\text{loc}}$ of homogeneous degree $k$. The following proposition is straightforward, and we leave it as an exercise.

\medskip 

\begin{proposition}
\label{prop:o}

Let us consider the following rational functions in place of \eqref{eqn:zeta}
\begin{equation}
\label{eqn:ozeta}
\ozeta_{ij}(x) = \left( \frac {1-x}{1-xq} \right)^{\delta_{ij}} \prod_{\alpha : j \rightarrow i} \frac {1 - \frac {t_{\alpha}}{xq}}{1 - \frac {t_{\alpha}}{x}}
\end{equation}
Because $\frac {\ozeta_{ij}(x)}{\ozeta_{ji}(x^{-1})} = \frac {\zeta_{ij}(x)}{\zeta_{ji}(x^{-1})}$, the map
\begin{equation}
\label{eqn:xi}
\Xi : \field [z_{i1}^{\pm 1},\dots,z_{in_i}^{\pm 1}]_{i \in I}^{\emph{sym}} \longrightarrow \frac {\field [z_{i1}^{\pm 1},\dots,z_{in_i}^{\pm 1}]_{i \in I}^{\emph{sym}} \cdot \prod^{i \in I}_{1\leq a \neq b \leq n_i} (z_{ia} - z_{ib})}{\prod^{i \in I}_{1\leq a \neq b \leq n_i} (z_{ia} - qz_{ib})\prod^{\alpha : i \rightarrow j}_{a \leq n_i, b \leq n_j} (z_{ia}t_{\alpha} - z_{jb})}
\end{equation}
of multiplication by
\begin{equation}
\label{eqn:division}
\frac{\prod^{i \in I}_{1 \leq a \neq b \leq n_i} \left(1 - \frac {z_{ia}}{z_{ib}} \right)}{q^{-\sum_{i \in I} {n_i \choose 2}} \prod^{i \in I}_{1 \leq a \neq b \leq n_i} \left(1 - \frac {z_{ia}q}{z_{ib}} \right)  \prod^{\alpha : i \rightarrow j}_{1 \leq a \leq n_i, 1\leq b \leq n_j,(i,a) \neq (j,b)} \left(1 - \frac {z_{ia}t_{\alpha}}{z_{jb}} \right)}
\end{equation}
intertwines the shuffle product $*$ of \eqref{eqn:mult} in the LHS with the shuffle product $\bar{*}$ in the RHS, the latter being defined using the functions \eqref{eqn:ozeta} instead of \eqref{eqn:zeta}.

\end{proposition}

\medskip 

\subsection{The localized double}
\label{sub:first double}

We define the (extended) double shuffle algebra as
\begin{equation}
\label{eqn:double loc}
\CS_{\text{loc}} = \CS^+_{\text{loc}} \otimes \field [\kappa^\pm_{i}, p_{i,\pm k}, c_{i, \pm k} ]_{i \in I, k > 0} \otimes \CS^-_{\text{loc}}
\end{equation}
The commuting elements $\kappa, p, c$ will be called Cartan elements. The multiplication between the three factors of \eqref{eqn:double loc} is controlled by relations \eqref{eqn:rel double 0}-\eqref{eqn:rel double 7} below 
\begin{equation}
\label{eqn:rel double 0}
c_{i, \pm k} \text{ are central}
\end{equation}
\begin{align}
&\kappa_i^+ E = E \kappa_i^+ \cdot q^{-\langle \bs^i, \bn\rangle} \label{eqn:rel double 1} \\
&\kappa_i^+ F = F \kappa_i^+ \cdot q^{\langle \bs^i, \bn\rangle}\label{eqn:rel double 2} \\
&\kappa_i^- E = E \kappa_i^- \cdot q^{\langle \bn, \bs^i\rangle} \label{eqn:rel double 3} \\
&\kappa_i^- F = F \kappa_i^- \cdot q^{\langle - \bn, \bs^i \rangle} \label{eqn:rel double 4}
\end{align} 
\begin{align}
&[p_{i,k}, E] = (1-q^k) (z_{i1}^k+\dots+z_{in_i}^k) E \label{eqn:rel double 5} \\
&[p_{i,k}, F] = (1-q^{-k}) (z_{i1}^k+\dots+z_{in_i}^k)F \label{eqn:rel double 6} 
\end{align}
for all $E(z_{i1},\dots,z_{in_i})_{i \in I} \in \CS^+_{\text{loc}}$, $F(z_{i1},\dots,z_{in_i})_{i \in I} \in \CS^-_{\text{loc}}$, $i \in I$ and $k \in \BZ \backslash 0$. Finally, we impose for all $i,j \in I$ and $k,\ell \in \BZ$ the relation
\begin{equation}
\label{eqn:rel double 7}
\Big[e_{i,k}, f_{j,\ell} \Big] = \frac {\delta_{ij}}{\gamma_i} \begin{cases} - \ph^+_{i,k+\ell} &\text{if }k+\ell > 0 \\ \ph^-_{i,0} - \ph^+_{i,0} &\text{if }k+\ell = 0 \\ \ph^-_{i,-k-\ell} &\text{if }k+\ell < 0 \end{cases}
\end{equation} 
where 
\begin{equation}
\label{eqn:gamma}
\gamma_i = \left(1-\frac 1q\right)\prod_{\alpha : i \rightarrow i} \left[ (1-t_\alpha)\left(1-\frac {t_{\alpha}}q\right) \right]
\end{equation}
and 
\begin{equation}
\label{eqn:def ph}
\ph^\pm_i(x) = \sum_{k=0}^{\infty} \frac {\ph^\pm_{i,k}}{x^{\pm k}} := 
\end{equation}
$$
:= \kappa_i^\pm \exp \left(\sum_{k=1}^{\infty} \frac {- c_{i,\pm k} + p_{i,\pm k}(1 + q^{\mp k})-  \sum_{\alpha : i \rightarrow j} p_{j,\pm k} t_\alpha^{\mp k} - \sum_{\alpha: j \rightarrow i} p_{j,\pm k} t_{\alpha}^{\pm k} q^{\mp k}}{kx^{\pm k}}\right)
$$
Formula \eqref{eqn:rel double 7} suffices to determine the commutator of arbitrary elements of $\CS^\pm_{\text{loc}}$, due to the generation result of Theorem \ref{thm:shuffle loc}. We note that $\CS_{\text{loc}}$ is a $\zz \times \BZ$ graded algebra, by extending the horizontal and vertical degrees of \eqref{eqn:deg 1}-\eqref{eqn:deg 2} by
\begin{align*} 
&\hdeg \kappa^\pm_{i} = \b0, \qquad \vdeg \kappa^\pm_{i} = 0 \\
&\hdeg p_{i,k} = \b0, \qquad \vdeg p_{i,k}=k \\
&\hdeg c_{i,k} = \b0, \qquad \vdeg c_{i,k}=k
\end{align*}
There is a natural shifted generalization of the double shuffle algebra, defined as follows: given $I$-tuples of integers $\ba = (a_i)_{i \in I}, \bb = (b_i)_{i \in I} \in \zz$, we define
\begin{equation}
\label{eqn:double shifted loc}
\CS^{\ba-\bb}_{\text{loc}} = \CS^+_{\text{loc}} \otimes \field [\kappa^\pm_{i}, p_{i,\pm k}, c_{i, \pm k} ]_{i \in I, k > 0}  \otimes \CS^-_{\text{loc}}
\end{equation}
given by imposing \eqref{eqn:rel double 0}-\eqref{eqn:rel double 6} and the following modification of \eqref{eqn:rel double 7}
\begin{equation}
\label{eqn:rel double 7 shifted}
\Big[e_{i,k}, f_{j,\ell} \Big] = \frac {\delta_{ij}}{\gamma_i} \Big( \delta_{k+\ell \leq b_i} \ph^-_{i,-k-\ell+b_i}  - \delta_{k+\ell \geq a_i} \ph^+_{i,k+\ell-a_i}\Big)  
\end{equation} 
The reason for the superscript $\ba - \bb$ in \eqref{eqn:double shifted loc} is that, up to isomorphism, the shifted double shuffle algebra only depends on the value of $\ba-\bb \in \zz$.

\medskip

\subsection{The coproduct} 
\label{sub:coproduct}

Let us recall the origin of relations \eqref{eqn:rel double 0}-\eqref{eqn:rel double 7} (and of relation \eqref{eqn:rel double 7 shifted} in the shifted case). We may endow
\begin{align*} 
&\CS^{\geq}_{\text{loc}} = \CS^+_{\text{loc}} \otimes \field [\kappa^+_{i}, p_{i, k}, c_{i,k} ]_{i \in I, k > 0}  \\ 
&\CS^{\leq}_{\text{loc}} = \field [\kappa^-_{i}, p_{i,- k}, c_{i,-k} ]_{i \in I, k > 0}  \otimes \CS^-_{\text{loc}} 
\end{align*}
with algebra structures by imposing those among relations \eqref{eqn:rel double 0}-\eqref{eqn:rel double 6} which pertain to the respective generators. The resulting algebras may be promoted to topological bialgebras via
\begin{equation}
\label{eqn:coproduct ph}
\Delta \left(\kappa^\pm_i \right) = \kappa^\pm_i \otimes \kappa^\pm_i, \ \Delta(p\text{ or }c)_{i,\pm k} = (p\text{ or }c)_{i,\pm k}\otimes 1 + 1 \otimes (p \text{ or }c)_{i,\pm k} 
\end{equation}
\begin{align}
\Delta(E) = \sum_{\b0 \leq \bm \leq \bn} \frac {\prod^{j \in I}_{m_j < b \leq n_j} \ph^+_j(z_{jb}) E(z_{i1},\dots , z_{im_i} \otimes z_{i,m_i+1}, \dots, z_{in_i})}{\prod^{i \in I}_{1\leq a \leq m_i} \prod^{j \in I}_{m_j < b \leq n_j} \zeta_{ji} \left( \frac {z_{jb}}{z_{ia}} \right)} \label{eqn:coproduct e} \\
\Delta(F) = \sum_{\b0 \leq \bm \leq \bn} \frac {F(z_{i1},\dots , z_{im_i} \otimes z_{i,m_i+1}, \dots, z_{in_i}) \prod^{i \in I}_{1 \leq a \leq m_i} \ph^-_i(z_{ia})}{\prod^{i \in I}_{1\leq a \leq m_i} \prod^{j \in I}_{m_j < b \leq n_j} \zeta_{ij} \left( \frac {z_{ia}}{z_{jb}} \right)} \label{eqn:coproduct f} 
\end{align}
for all $E(z_{i1},\dots,z_{in_i})_{i \in I} \in \CS^+_{\text{loc}}$ and $F(z_{i1},\dots,z_{in_i})_{i \in I} \in \CS^-_{\text{loc}}$. To make sense of the right-hand sides of formulas \eqref{eqn:coproduct e} and \eqref{eqn:coproduct f}, we expand the denominator as a power series in the range $|z_{ia}| \ll |z_{jb}|$, and place all the powers of $z_{ia}$ to the left of the $\otimes$ sign and all the powers of $z_{jb}$ to the right of the $\otimes$ sign (for all $i,j \in I$, $1 \leq a \leq m_i$, $m_j < b \leq n_j$). In particular, the elements \eqref{eqn:special} and \eqref{eqn:special transpose} satisfy
\begin{align}
&\Delta(e_{\bn,g}) = \sum_{\b0 \leq \bm \leq \bn} \ph^+ e_{\bm,g_+'} \otimes e_{\bn-\bm,g_+''} \label{eqn:coproduct special e} \\
&\Delta(f_{\bn,g}) = \sum_{\b0 \leq \bm \leq \bn} f_{\bm,g_-'} \otimes f_{\bn-\bm,g_-''}\ph^- \label{eqn:coproduct special f}
\end{align}
where $g_\pm' \otimes g_\pm''$ (respectively $\ph^\pm$) denote infinite tensors of Laurent polynomials (respectively products of Cartan elements) that one can calculate starting from $g$ by performing the power series expansions in \eqref{eqn:coproduct e}-\eqref{eqn:coproduct f}.

\medskip 

\subsection{The Drinfeld double construction} 
\label{sub:drinfeld double}

There exists a pairing
\begin{equation}
\label{eqn:pairing loc}
\CS^{\geq}_{\text{loc}} \otimes \CS^{\leq}_{\text{loc}} \xrightarrow{\langle \cdot , \cdot \rangle} \field 
\end{equation}
which satisfies the usual bialgebra properties for all $e,e',e'' \in \CS^{\geq}_{\text{loc}}$, $f,f',f'' \in \CS^{\leq}_{\text{loc}}$
\begin{align*} 
&\langle e, f'f'' \rangle = \langle e_1,f'\rangle \langle e_2,f''\rangle \\
&\langle e'e'', f \rangle = \langle e',f_2\rangle \langle e'',f_1\rangle 
\end{align*}
(we use Sweedler notation $\Delta(e) = e_1 \otimes e_2$ and $\Delta(f) = f_1\otimes f_2$ throughout the present paper) and is completely determined by the conditions 
\begin{equation} 
\label{eqn:pairing kappa}
\Big \langle \kappa_i^+, \kappa_j^- \Big \rangle = q^{-\langle \bs^i,\bs^j \rangle}
\end{equation} 
\begin{equation} 
\label{eqn:pairing p}
\Big \langle p_{i,k}, p_{j,-k} \Big \rangle = \delta_{ij}(q^{-k}-q^k) + \sum_{\alpha : i \rightarrow j} t_{\alpha}^{-k}(q^k-1) + \sum_{\alpha : j \rightarrow i} t_\alpha^k(1-q^{-k})
\end{equation} 
\begin{equation} 
\label{eqn:pairing ef}
\Big \langle E,F \Big \rangle = \prod_{i\in I} \gamma_i^{-n_i} \cdot \int  \frac {E(z_{i1},\dots,z_{in_i})F(z_{i1},\dots,z_{in_i})}{\prod_{(i,a) \neq (j,b)} \zeta_{ij}\left(\frac {z_{ia}}{z_{jb}} \right)} \prod_{i \in I} \prod_{a=1}^{n_i} \frac {dz_{ia}}{2\pi i z_{ia}}
\end{equation} 
for all $E(z_{i1},\dots,z_{in_i})_{i \in I} \in \CS^+_{\text{loc}}$ and $F(z_{i1},\dots,z_{in_i})_{i \in I} \in \CS^-_{\text{loc}}$. All pairings between $\kappa,p,c,E,F$ other than the ones above vanish; in particular, we have $\langle E, F \rangle = 0$ unless $\hdeg E + \hdeg F = \b0$. The definition of the integral in \eqref{eqn:pairing ef} can be found in \cite{Wheel}, and will not be used in the present paper. With this in mind, the algebra \eqref{eqn:double loc} is none other than the Drinfeld double
$$
\CS_{\text{loc}} = \CS^{\geq}_{\text{loc}} \otimes \CS^{\leq}_{\text{loc}}
$$
defined with respect to the coproduct and pairing above, i.e. we have
\begin{equation}
\label{eqn:dd} 
\langle e_2,f_2\rangle e_1 f_1 = \langle e_1,f_1 \rangle f_2e_2
\end{equation}
for all $e \in \CS^{\geq}_{\text{loc}}$ and $f \in \CS^{\leq}_{\text{loc}}$. Similarly, the shifted algebra \eqref{eqn:double shifted loc} is
$$
\CS^{\ba-\bb}_{\text{loc}} = \CS^{\geq}_{\text{loc}} \otimes \CS^{\leq}_{\text{loc}}
$$
where we modify the Drinfeld double relation \eqref{eqn:dd} as follows
\begin{equation}
\label{eqn:dd shifted} 
\langle e_2,T^{\ba}(f_2)\rangle e_1 f_1 = \langle T^{-\bb}(e_1),f_1 \rangle f_2e_2 
\end{equation}
for all $e \in \CS^{\geq}_{\text{loc}}$ and $f \in \CS^{\leq}_{\text{loc}}$, with respect to the shift automorphisms
\begin{align*} 
&T^{\br} : \CS^{\geq}_{\text{loc}} \mapsto \CS^{\geq}_{\text{loc}}, \qquad T^{\br}(E) = E \prod_{i \in I} \prod_a z_{ia}^{r_i}, \quad \ T^{\br}(\kappa \text{ or }p \text{ or }c) = (\kappa \text{ or }p \text{ or }c) \\
&T^{\br} : \CS^{\leq}_{\text{loc}} \mapsto \CS^{\leq}_{\text{loc}}, \qquad T^{\br}(F) = F \prod_{i \in I} \prod_a z_{ia}^{-r_i}, \quad T^{\br}(\kappa \text{ or }p \text{ or }c) = (\kappa \text{ or }p \text{ or }c)
\end{align*} 
for all $\br \in \zz$.

\medskip 

\subsection{The integral double}
\label{sub:integral double}

We define $\CS \subset \CS_{\text{loc}}$ and more generally
\begin{equation}
\label{eqn:integral version}
\CS^{\ba-\bb} \subset \CS^{\ba-\bb}_{\text{loc}}
\end{equation}
as the $\ring$-subalgebra generated by

\medskip 

\begin{itemize}[leftmargin=*]

\item the integral shuffle algebras $\CS^\pm \subset \CS^\pm_{\text{loc}}$, and 

\medskip 

\item the subalgebra $\ring [p_{i,\pm k}]_{i \in I, k > 0}$ of Cartan elements.

\end{itemize}

\medskip 

\noindent We do not necessarily expect a triangular decomposition $\CS \neq \CS^+ \otimes (\text{Cartan}) \otimes \CS^-$, for the following reason. The Drinfeld double relations \eqref{eqn:dd} can be rewritten as
\begin{equation}
\label{eqn:dd equivalent}
fe = \langle e_1,S(f_1)\rangle e_2f_2 \langle e_3,f_3 \rangle
\end{equation} 
where we write $\Delta^{(2)}(e) = e_1 \otimes e_2 \otimes e_3$, $\Delta^{(2)}(f) = f_1 \otimes f_2 \otimes f_3$ \footnote{In the shifted case \eqref{eqn:dd shifted}, one needs to replace $f_3$ by $T^{\ba}(f_3)$ and $e_1$ by $T^{-\bb}(e_1)$.}. In particular, for the generators \eqref{eqn:special} and \eqref{eqn:special transpose} of the integral shuffle algebras, we have for all $\bm,\bn,g,h$
\begin{equation}
\label{eqn:commutator}
f_{\bm,h} e_{\bn,g} = \mathop{\sum_{\bm=\bm_1+\bm_2+\bm_3}}_{\bn = \bn_1+\bn_2+\bn_3} \langle \ph^+ e_{\bn_1,g_1}, S(f_{\bm_1,h_1}) \rangle \ph^+ e_{\bn_2,g_2}  f_{\bm_2,h_2} \ph^- \langle e_{\bn_3,g_3},f_{\bm_3,h_3} \ph^- \rangle
\end{equation}
where $g_1 \otimes g_2 \otimes g_3$ and $h_1 \otimes h_2 \otimes h_3$ denote certain infinite sums of tensors of Laurent polynomials that can be calculated from $g$ and $h$ using formulas \eqref{eqn:coproduct special e}-\eqref{eqn:coproduct special f} (while the sums are infinite are infinite, only finitely many of the airings above will be non-zero, resulting of the right-hand side of the expression \eqref{eqn:commutator} being a finite sum). While the aforementioned tensors and the pairings in the right-hand side of \eqref{eqn:commutator} may be calculated in practice, the result may have coefficients in $\field$ and not in $\ring$.

\bigskip 

\section{The Coulomb branch}
\label{sec:coulomb}

\medskip

\subsection{The BFN construction and difference operators}
\label{sub:bfn}

Consider a quiver $Q$ as in Subsection \ref{sub:quivers}, with equivariant parameters as in \eqref{eqn:parameters}-\eqref{eqn:opposite parameters}. Given any dimension vector $\bd = (d_i \geq 0)_{i \in I} \in \nn$ and framing parameters
\begin{equation}
\label{eqn:framing}
\underbrace{\sigma_{i1},\dots,\sigma_{ik_i}}_{\text{corresponding to arrows }i \rightarrow \square} \quad \text{and} \quad \underbrace{\tau_{i1},\dots,\tau_{i\ell_i}}_{\text{corresponding to arrows } \square \rightarrow i}
\end{equation}
Let $\bk = (k_i)_{i \in I}$ and $\bell = (\ell_i)_{i \in I}$. Our $q$, $t_\alpha$, $\sigma_{is}, \tau_{it}$ are the $q^2$, $qu(\alpha)$, $z_{\square}q^{\pm 1}$ of \cite{SS}. One of our main objects of study is the $K$-theoretic \emph{Coulomb branch algebra} 
$$
\CA_{\bd|\bk,\bell} \text{ over the ring } R[w_{i1}^{\pm 1},\dots,w_{id_i}^{\pm 1}]^{\sym}_{i \in I}
$$
defined by Braverman, Finkelberg and Nakajima in \cite{BFNQuiver}, and studied in numerous other works. We will not need to recall the definition of $\CA_{\bd|\bk,\bell}$, and refer the reader to Subsection 2.1 of \cite{SS} for an overview in our conventions. The main thing we will use about the Coulomb branch algebra in the present paper is that there exists an injective $\ring$-algebra homomorphism 
\begin{equation}
\label{eqn:coulomb hom}
\CA_{\bd|\bk,\bell} \hookrightarrow \DD_{\bd|\bk,\bell}
\end{equation}
constructed using the localization procedure in \cite{BFN, BFNQuiver, KWWY}, where 
\begin{equation}
\label{eqn:difference operators}
\DD_{\bd|\bk,\bell} = \frac {\field(w_{ia})_{i \in I, a \in \{1,\dots,d_i\}} \otimes \field[D^{\pm 1}_{ia}]_{i \in I, a \in \{1,\dots,d_i\}}}{D_{ia} w_{jb} - q^{\delta_{ij} \delta_{ab}} w_{jb}D_{ia}} [\sigma^{\pm 1}_{i1},\dots,\sigma^{\pm 1}_{ik_i}, \tau_{i1}^{\pm 1},\dots,\tau^{\pm 1}_{i\ell_i}]_{i \in I}
\end{equation}
is an explicit algebra of $q$-difference operators. 

\medskip

\subsection{The Finkelberg-Frassek-Tsymbaliuk homomorphism}
\label{sub:gklo}

Motivated by the work \cite{GKLO} on Yangians and difference operators, Finkelberg-Tsymbaliuk (\cite{FT}, type $A$), Frassek-Tsymbaliuk (\cite{FrT}, classical types) and Tsymbaliuk (\cite{Tsymbaliuk}, general type) constructed a homomorphism from the shifted double shuffle algebra to $\DD_{\bd|\bk,\bell}$. Specifically, Theorem 5.7 in \loccit provides algebra homomorphisms
\begin{equation}
\label{eqn:tsymbaliuk}
\Ts : \CV^\pm_{\text{loc}} \rightarrow \DD_{\bd|\bk,\bell}
\end{equation}
defined for any $E(z_{i1},\dots,z_{in_i})_{i \in I} \in \CV^+_{\text{loc}}$ and $F(z_{i1},\dots,z_{in_i})_{i \in I} \in \CV^-_{\text{loc}}$ by
\begin{equation}
\label{eqn:action e}
E \mapsto (1-q^{-1})^{-|\bn|} \sum^{I-\text{compositions }P}_{\{n_i = n_i^{(1)}+ \dots + n_i^{(d_i)}\}_{i \in I}} \text{Res}^+_P(\Xi(E))
\end{equation} 
$$
\prod^{(i,a)}_{0 \leq c < n_i^{(a)}} 
\left[ \prod^{\alpha : i \rightarrow j}_{(j,b) \neq (i,a)} \left(1 - \frac {w_{ia}q^c t_\alpha}{w_{jb}}\right) \prod_{s=1}^{k_i} \left(1-\frac {w_{ia}q^c}{\sigma_{is}} \right) \prod_{b \neq a} \left(1 - \frac {w_{ib}}{w_{ia}q^c} \right)^{-1} \right] \prod_{(i,a)} D_{ia}^{n_i^{(a)}}
$$
\begin{equation}
\label{eqn:action f}
F \mapsto (1-q^{-1})^{-|\bn|} \sum^{I-\text{compositions }P}_{\{n_i = n_i^{(1)}+ \dots + n_i^{(d_i)}\}_{i \in I}} \text{Res}^-_P(\Xi(F))
\end{equation} 
$$
\prod^{(i,a)}_{0 \leq c < n_i^{(a)}} 
\left[ \prod^{\alpha : j \rightarrow i}_{(j,b) \neq (i,a)} \left(1 - \frac {w_{jb}t_\alpha}{w_{ia}q^c}\right) \prod_{t=1}^{\ell_i} \left(1-\frac {\tau_{it}}{w_{ia}q^{c}} \right) \prod_{b \neq a} \left(1 - \frac {w_{ia}q^c}{w_{ib}} \right)^{-1} \right] \prod_{(i,a)} D_{ia}^{-n_i^{(a)}}
$$
In the formulas above, we recall that $\Xi(E)$ and $\Xi(F)$ of \eqref{eqn:xi} are rational functions in $z_{ia}$ with at most simple poles at $z_{ia} = q z_{ib}$. For such a rational function $G(z_{i1},\dots,z_{in_i})_{i \in I}$ and any $I$-composition $P = \{n_i = n_i^{(1)}+ \dots + n_i^{(d_i)}\}_{i \in I}$, we define
\begin{equation}
\label{eqn:residue}
\text{Res}_P(G) = 
\end{equation}
$$
\left( \mathop{\text{Res}}_{z_{i,\nu_i^{(a)}} = q^{n_i^{(a)}-1}z_{i,\nu_i^{(a-1)}+1}} \cdots \mathop{\text{Res}}_{z_{i,\nu_i^{(a-1)}+2} = qz_{i,\nu_i^{(a-1)}+1}} \frac {G}{z_{i,\nu_i^{(a)}} \dots z_{i,\nu_i^{(a-1)}+2}}  \right)_{i \in I, a \leq d_i}
$$
where $\nu_i^{(a)} = n_i^{(1)}+\dots+n_i^{(a)}$ for all $a \in \{1,\dots,d_i\}$. Then in \eqref{eqn:action e}-\eqref{eqn:action f}, one sets
\begin{align*}
&\text{Res}^+_P(G) = \frac {\text{Res}_P(G)|_{z_{i,\nu_i^{(a-1)}+1 = w_{ia}}}}{\prod_{i \in I} \prod_{a=1}^{d_i} \left[(1-q^{-1})\dots(1-q^{-n_i^{(a)}+1})  \prod_{\alpha : i \rightarrow i} \frac {(1-t_\alpha)^{n_i^{(a)}-1}}{(1-t_{\alpha}q)\dots (1-t_\alpha q^{n_i^{(a)}-1})}\right]} \\
&\text{Res}^-_P(G) = \frac {\text{Res}_P(G)|_{z_{i,\nu_i^{(a-1)}+1 = w_{ia} q^{-n_i^{(a)}}}}}{\prod_{i \in I} \prod_{a=1}^{d_i} \left[(1-q^{-1})\dots(1-q^{-n_i^{(a)}+1})  \prod_{\alpha : i \rightarrow i} \frac {(1-t_\alpha)^{n_i^{(a)}-1}}{(1-t_{\alpha}q)\dots (1-t_\alpha q^{n_i^{(a)}-1})}\right]}
\end{align*}
In the following result, we construct $I$-tuples $\ba = (a_i)_{i \in I}$ and $\bb = (b_i)_{i \in I}$ by
\begin{align}
&a_i = d_i - \sum_{j \in I} d_j \Big|\text{arrows }i \rightarrow j\Big| - k_i \label{eqn:a} \\
&b_i = -d_i + \sum_{j \in I} d_j \Big|\text{arrows }j \rightarrow i\Big| + \ell_i \label{eqn:b}
\end{align}
We denote by $\CS^{\bd^\vee-\bk-\bell}_{\text{loc}}$ the shifted double shuffle algebra of \eqref{eqn:double shifted loc} for the $I$-tuples $\ba,\bb$ defined above, where $\bd^\vee$ is defined in \eqref{eqn:vee} and $\bk,\bell$ are determined by \eqref{eqn:framing}.

\medskip 

\begin{theorem}
\label{thm:tsymbaliuk}

(\cite{FT, FrT, Tsymbaliuk}) Formulas \eqref{eqn:action e}-\eqref{eqn:action f} yield algebra homomorphisms \eqref{eqn:tsymbaliuk}. The restriction of these homomorphisms to the subalgebras $\CS^\pm_{\emph{loc}} \subset \CV^\pm_{\emph{loc}}$, together with the assignments
\begin{align} 
&\kappa^+_i \mapsto \frac {\prod_{\alpha : i\rightarrow j} \prod_{b=1}^{d_j} \left(-\frac {t_{\alpha}}{w_{jb}} \right) \prod_{s=1}^k \left(-\frac 1{\sigma_{is}}\right)}{\prod_{b=1}^{d_i} \left(-\frac {q}{w_{ib}} \right)} \label{eqn:assignment 1} \\ 
&\kappa^-_i \mapsto \frac {\prod_{\alpha : j\rightarrow i} \prod_{b=1}^{d_j} \left(-\frac {w_{jb}t_{\alpha}}q \right) \prod_{t=1}^\ell \left(- \tau_{it}\right)}{\prod_{b=1}^{d_i} \left(- w_{ib} \right)} \label{eqn:assignment 2}
\end{align}
and
\begin{equation}
\label{eqn:assignment 3}
p_{i,k} \mapsto w_{i1}^k+\dots + w_{id_i}^k
\end{equation}
\begin{equation}
\label{eqn:assignment 4}
c_{i,k} \mapsto \sum_{s=1}^{k_i} \sigma_{is}^k + \sum_{t=1}^{\ell_i} \tau_{it}^k 
\end{equation}
glue to a morphism 
\begin{equation}
\label{eqn:tsymbaliuk theorem}
\eTs : \CS^{\bd^\vee-\bk-\bell}_{\emph{loc}} \rightarrow \DD_{\bd|\bk,\bell} 
\end{equation}
that we dub the \emph{Finkelberg-Frassek-Tsymbaliuk homomorphism}. 

\end{theorem}

\medskip 

\begin{remark} The aforementioned connection between double shuffle algebras and difference operators, which we attribute to Finkelberg-Frassek-Tsymbaliuk, is also compatible with a very important connection between Yangians/quantum loop groups and Coulomb branch algebras (\cite{BFNQuiver, BDG, KWWY}).
\end{remark}

\medskip 

\noindent The assignments for the Cartan elements in Theorem \ref{thm:tsymbaliuk} imply that the generating series $x^{-a_i}\ph^+_i(x)$ and $x^{-b_i}\ph^-_i(x)$ of \eqref{eqn:def ph} both correspond to 
$$
\frac {\prod^{(j,b)}_{\alpha : i \rightarrow j} \left(1 - \frac {xt_{\alpha}}{w_{jb}}\right) \prod^{(j,b)}_{\alpha : j \rightarrow i} \left(1 - \frac {t_{\alpha}w_{jb}}{xq} \right) \prod_{s=1}^{k_i} \left(1 - \frac x{\sigma_{is}} \right) \prod_{t=1}^{\ell_i} \left(1 - \frac {\tau_{it}}x \right)}{\prod_{i,a} \left(1 - \frac {w_{ia}}x\right)\left(1 - \frac {xq}{w_{ia}}\right)} \in \DD_{\bd|\bk,\bell}(x)
$$
but expanded near $x \sim \infty$ and $x \sim 0$, respectively.

\medskip 

\subsection{Coulomb branches}
\label{sub:iso}

The following Proposition will be key, and it is a natural analogue of Lemma 2.12 of \cite{Tsymbaliuk}.

\medskip 

\begin{proposition}
\label{prop:key}

For the shuffle elements $e_{\bn,g}, f_{\bn,g}$ of \eqref{eqn:special} and \eqref{eqn:special transpose}, we have
$$
\eTs(e_{\bn,g}) = \begin{cases} E_{\bn,g} &\text{if } \bn \leq \bd \\ 0&\text{otherwise} \end{cases}, \qquad \eTs(f_{\bn,g}) = \begin{cases} F_{\bn,g} &\text{if } \bn \leq \bd \\ 0&\text{otherwise} \end{cases},
$$
where for any $\b0 \leq \bn \leq \bd$, we set
\begin{equation}
\label{eqn:special e}
E_{\bn,g} = q^{-\sum_{i \in I} {n_i \choose 2}} \sum_{\{A_i \subset \{1,\dots,d_i\}, |A_i| = n_i\}_{i \in I}} g(w_{ia})_{i \in I, a \in A_i}
\end{equation} 
$$
\prod_{i,j \in I} \prod^{\alpha : i \rightarrow j}_{a \in A_i, b \notin A_j} \left(1 - \frac {w_{ia} t_\alpha}{w_{jb}}\right) \prod_{i \in I} \prod_{a \in A_i} \prod_{s=1}^{k_i} \left(1-\frac {w_{ia}}{\sigma_{is}} \right) \prod_{i \in I} \prod_{a \in A_i, b \notin A_i} \left(1 - \frac {w_{ib}}{w_{ia}} \right)^{-1} \prod_{i \in I}  \prod_{a \in A_i} D_{ia}
$$
\begin{equation}
\label{eqn:special f}
F_{\bn,g} = q^{-\sum_{i \in I} {n_i \choose 2}} \sum_{\{A_i \subset \{1,\dots,d_i\}, |A_i| = n_i\}_{i \in I}} g(w_{ia}q^{-1})_{i \in I, a \in A_i}
\end{equation} 
$$
\prod_{i,j \in I} \prod^{\alpha : j \rightarrow i}_{a \in A_i, b \notin A_j} \left(1 - \frac {w_{jb}t_\alpha}{w_{ia}}\right) \prod_{i \in I} \prod_{a \in A_i} \prod_{t=1}^{\ell_i} \left(1-\frac {\tau_{it}}{w_{ia}} \right) \prod_{i \in I} \prod_{a \in A_i, b \notin A_i} \left(1 - \frac {w_{ia}}{w_{ib}} \right)^{-1}  \prod_{i \in I} \prod_{a \in A_i} D_{ia}^{-1}
$$

\end{proposition}

\medskip 

\begin{proof} \emph{(sketch)} Because the shuffle elements $e_{\bn,g}$ and $f_{\bn,g}$ are multiples of $z_{ia}q-z_{ib}$ for all $a \neq b$, the residues \eqref{eqn:residue} of $\Xi(e_{\bn,g})$ and $\Xi(f_{\bn,g})$ vanish if any part of the $I$-composition $P$ is greater than 1. On the other hand, if all parts of $P$ are equal to 1, this means that we must necessarily have $\bn \leq \bd$ and then the right-hand sides of \eqref{eqn:action e}-\eqref{eqn:action f} are easily seen to match the right-hand sides of \eqref{eqn:special e}-\eqref{eqn:special f} on the nose.

\end{proof}

\medskip 

\noindent As shown in \cite{BFN, DK, FT, SS, Weekes}, the image of the homomorphism $\CA_{\bd|\bk,\bell} \hookrightarrow \DD_{\bd|\bk,\bell}$ of \eqref{eqn:coulomb hom} is generated over $\ring$ by 
$$
\Big\{ E_{\bn,g}, F_{\bn,g} \Big\}_{\b0 \leq \bn \leq \bd, g \in R[w_{i1}^{\pm 1},\dots,w_{id_i}^{\pm 1}]^{\sym}_{i \in I}}
$$
Because of Proposition \ref{prop:key}, we conclude that the Finkelberg-Frassek-Tsymbaliuk homomorphism descends to a surjective $\ring$-algebra homomorphism
\begin{equation}
\label{eqn:hom}
\CS^{\bd^\vee - \bk - \bell} \twoheadrightarrow \CA_{\bd|\bk,\bell}
\end{equation}
where we stress the fact that the algebra in the domain of \eqref{eqn:hom} is the integral version \eqref{eqn:integral version} and not the localized version \eqref{eqn:double shifted loc}.

\medskip 

\begin{proof} \emph{of Theorem \ref{thm:main}:} Since we identified $\CS^+\cong \kha$ in \eqref{eqn:k shuffle int} and defined the shifted Drinfeld double of the loop-nilpotent $K$-HA in \eqref{eqn:rigorous main} to be the shifted double shuffle algebra \eqref{eqn:integral version}, then \eqref{eqn:hom} is precisely the sought-for homomorphism \eqref{eqn:main}. 
\end{proof}

\bigskip

\section{Appendix A: shuffle algebras}
\label{sec:a}

\medskip

\subsection{Multiplicativity}

In the present subsection, we prove the following.

\medskip 

\begin{proposition}
\label{prop:algebras}

The set $\CS^+$ of Definition \ref{def:shuffle int} is an algebra.

\end{proposition}

\medskip 

\begin{proof} Let us consider shuffle elements
$$
E(z_{i1},\dots,z_{in_i})_{i \in I} \quad \text{and} \quad E'(z_{i1},\dots,z_{in_i'})_{i \in I}
$$
which satisfy the divisibility conditions of Definition \ref{def:shuffle int}, and we will prove that their shuffle product $E*E'$ also satisfies the same divisibility conditions. In fact, we will prove the stronger statement that every summand in \eqref{eqn:mult} corresponding to every possible shuffle of the variables $\{z_{i1},\dots,z_{in_i}\}_{i \in I}$ and $\{z_{i,n_i+1},\dots,z_{i,n_i+n_i'}\}_{i \in I}$ satisfies the required divisibility conditions. Thus, let us take any $I$-composition
$$
\left\{n_i+n_i' = k_i^{(1)} + \dots + k_i^{(d_i)} \right\}_{i \in I}
$$
and specialize the variables of
\begin{equation}
\label{eqn:summand}
E(z_{i1},\dots,z_{in_i})_{i \in I} E'(z_{i,n_i+1},\dots,z_{i,n_i+n_i'})_{i \in I} \prod_{i,j \in I} \prod_{a=1}^{n_i} \prod_{b={n_j+1}}^{n_j+n_j'} \zeta_{ij} \left(\frac {z_{ia}}{z_{jb}} \right)
\end{equation}
to the geometric progressions
$$
\left\{x_{ia},x_{ia}q,\dots,x_{ia}q^{k_i^{(a)}-1} \right\}_{i \in I, a \in \{1,\dots,d_i\}}
$$
Recall that
$$
\zeta_{ij}(x) \sim \left(\frac {1-xq^{-1}}{1-x}\right)^{\delta_{ij}} \short (1-xt_{\alpha})
$$
where $\sim$ means that the two sides of the equation are equal up to some monomials that will not affect the upcoming argument. It is clear from the formula above that $\zeta_{ii}(q) = 0$, so the only specializations which are non-zero in \eqref{eqn:summand} are those for which
\begin{align*} 
&x_{ia},\dots,x_{ia}q^{\ell_i^{(a)}-1} \qquad \text{ are plugged into the variables of }E \\
&x_{ia}q^{\ell_i^{(a)}},\dots,x_{ia}q^{k_i^{(a)}-1} \ \text{ are plugged into the variables of }E'
\end{align*}
for some $\{0 \leq \ell_i^{(a)} \leq k_i^{(a)}\}_{i,a}$. But then the specialization of \eqref{eqn:summand} is divisible by
\begin{equation}
\label{eqn:1}
\prod_{i \in I} \prod_{a=1}^{d_i} \left[(1-q^{-1})\dots (1-q^{-\ell_i^{(a)}})\right] \prod_{\alpha : i \rightarrow j} \mathop{\prod_{1 \leq a \leq d_i}}_{1 \leq b \leq d_j} \prod_{c \in \BZ} \left(1- \frac {x_{ia}t_{\alpha}}{x_{jb}q^{c}} \right)^{\chi_{\ell_i^{(a)},\ell_j^{(b)}}(c)} 
\end{equation}
from the divisibility property of $E$, 
\begin{multline}
\label{eqn:2}
\prod_{i \in I} \prod_{a=1}^{d_i} \left[(1-q^{-1})\dots (1-q^{-k_i^{(a)}+\ell_i^{(a)}})\right] \\ \prod_{\alpha : i \rightarrow j} \mathop{\prod_{1 \leq a \leq d_i}}_{1 \leq b \leq d_j} \prod_{c \in \BZ} \left(1- \frac {x_{ia}t_{\alpha}}{x_{jb}q^{c}} \right)^{\chi_{k_i^{(a)}-\ell_i^{(a)},k_j^{(b)}-\ell_j^{(b)}}(c+\ell_i^{(a)}-\ell_j^{(b)})} 
\end{multline}
from the divisibility property of $E'$, and
\begin{multline}
\label{eqn:3}
\prod_{i \in I} \prod_{a=1}^{d_i} \prod_{r=0}^{\ell_i^{(a)}-1} \prod_{s=\ell_i^{{(a)}}}^{k_i^{(a)}-1} \frac {1-q^{r-s-1}}{1-q^{r-s}} \prod_{\alpha : i \rightarrow j} \mathop{\prod_{1 \leq a \leq d_i}}_{1 \leq b \leq d_j} \\ \left[  \prod_{r=0}^{\ell_i^{(a)}-1} \prod_{s=\ell_j^{(b)}}^{k_j^{(b)}-1} \left(1 - \frac {x_{ia}q^r t_{\alpha}}{x_{jb}q^s} \right) \prod_{r=\ell_i^{(a)}}^{k_i^{(a)}-1} \prod_{s=0}^{\ell_j^{(b)}-1} \left(1 - \frac {x_{ia}q^r t_{\alpha}}{x_{jb}q^{s+1}} \right) \right]
\end{multline}
from the zeta factors which appear in \eqref{eqn:summand}. Multiplying out the factors in \eqref{eqn:1}, \eqref{eqn:2}, \eqref{eqn:3} precisely produces the divisibility factor \eqref{eqn:factor}, due to the elementary inequality
\begin{equation}
\label{eqn:elementary identity}
\chi_{k,k'}(c) \leq \chi_{\ell,\ell'}(c) + \chi_{k-\ell,k'-\ell'}(c+\ell-\ell') + 
\end{equation}
$$
\Big|\{0 \leq r < \ell, \ell' \leq s < k' \text{ s.t. }c = s-r\}\Big| + \Big|\{\ell \leq r < k, 0 \leq s < \ell' \text{ s.t. }c = s+1-r\}\Big|
$$
that holds for all $0 \leq \ell \leq k$, $0 \leq \ell' \leq k'$. To prove \eqref{eqn:elementary identity}, we will invoke the following explicit description of the $\chi$ function
\begin{equation}
\label{eqn:explicit chi}
\chi_{k,k'}(c) = \begin{cases} \Big| \{ (r,s) \in \{0,\dots,k-1\} \times \{0,\dots, k'-1\} \text{ s.t. } s-r = c  \} \Big| &\text{if }c > 0 \\ \Big| \{ (r,s) \in \{0,\dots,k-1\} \times \{0,\dots, k'-1\} \text{ s.t. } s+1-r = c \} \Big| &\text{if }c \leq 0 \end{cases}
\end{equation}
which we leave to the reader. Therefore, we have
\begin{align*}
&\text{LHS of \eqref{eqn:elementary identity}} - \chi_{\ell,\ell'}(c) = \begin{cases} \Big| \{ (r,s) \in A \text{ s.t. } s-r = c \} \Big| &\text{if }c > 0 \\ \Big| \{ (r,s) \in A \text{ s.t. } s+1-r = c \} \Big| &\text{if }c \leq 0 \end{cases} \\
&\text{RHS of \eqref{eqn:elementary identity}} - \chi_{\ell,\ell'}(c) = \begin{cases} \Big| \{ (r,s) \in B' \text{ s.t. } s-r = c \} \Big| &\text{if }c > \ell' - \ell \\ \Big| \{ (r,s) \in B'' \text{ s.t. } s+1-r = c \} \Big| &\text{if }c \leq \ell' - \ell \end{cases}
\end{align*}
where we set $A = \{0,\dots,k-1\} \times \{0,\dots, k'-1\} \ \backslash \ \{0,\dots,\ell-1\} \times \{0,\dots, \ell'-1\}$, $B' = \{0,\dots,k-1\} \times \{\ell',\dots,k'-1\}$ and $B'' = \{\ell,\dots,k-1\} \times \{0,\dots,k'-1\}$. While it is true that $A \supset B'$ and $A \supset B''$, it is clear that for any $(r,s) \in A \backslash B'$ we have $s-r < \ell'-\ell$ and for any $(r,s) \in A \backslash B''$ we have $s+1-r > \ell' - \ell + 1$. Therefore, the difference between the RHS and LHS of \eqref{eqn:elementary identity} is
$$
\begin{cases} \Big| \{ (r,s) \in B'' \text{ s.t. } s+1-r = c \} \Big| -  \Big| \{ (r,s) \in A \text{ s.t. } s-r = c \} \Big| &\text{if }0 <  c \leq \ell'-\ell \\ \Big| \{ (r,s) \in B' \text{ s.t. } s-r = c \} \Big| -  \Big| \{ (r,s) \in A \text{ s.t. } s+1-r = c \} \Big| &\text{if }\ell'-\ell < c \leq 0 \\ 0 &\text{otherwise}  \end{cases}
$$
It is elementary to show that the numbers above are all 0 or 1, which proves \eqref{eqn:elementary identity}.

\end{proof}

\medskip 

\subsection{Generators}

\medskip 

Having shown that $\CS^+$ is an algebra, we will prove that it coincides with the $\ring$-subalgebra of $\CV^+$ generated by the elements $e_{\bn,g}$ of \eqref{eqn:special}, as $\bn$ runs over $\nn$ and $g$ over color-symmetric Laurent polynomials with coefficients in $R$.

\medskip 

\begin{proof} \emph{of Theorem \ref{thm:shuffle int}:} Denote the $\ring$-subalgebra generated by the $e_{\bn,g}$'s by
\begin{equation}
\label{eqn:subalgebra generated by}
\bar{\CS}^+ \subset \CV^+
\end{equation}
The $e_{\bn,g}$'s lie in the set $\CS^+$ of Definition \ref{def:shuffle int} because if an $I$-composition $P$ has at least one part of size $>1$, the polynomial $\text{Spec}_P(e_{\bn,g})$ of \eqref{eqn:specialization} vanishes identically (meanwhile, if $P$ is the finest $I$-composition with all parts equal to 1, the divisibility condition of Definition \ref{def:shuffle int} is trivial). Since we already showed in Proposition \ref{prop:algebras} that $\CS^+$ is an $\ring$-subalgebra of $\CV^+$, we conclude that
\begin{equation}
\label{eqn:star}
\bar{\CS}^+ \subseteq \CS^+
\end{equation}
For the opposite inclusion, we will consider the specialization maps
$$
\text{Spec}_P : \CS^+ \rightarrow \ring[x_{i1}^{\pm 1},\dots,x_{id_i}^{\pm 1}]_{i \in I}
$$
associated to all $I$-partitions $P$, in the notation \eqref{eqn:composition}. Consider the total lexicographic order on partitions of the same size
$$
(k_1 \geq k_2 \geq \dots > 0) > (\ell_1 \geq \ell_2 \geq \dots > 0)
$$
if there exists $s$ such that $k_1=\ell_1$, \dots, $k_{s-1} = \ell_{s-1}$ and $k_s > \ell_s$. This induces a partial order on the set of $I$-partitions, where we set $P \geq P'$ if the $i$-th part of $P$ is greater than or equal to the $i$-th part of $P'$ for all $i \in I$. With this in mind, let
$$
\CS^+_P = \bigcap_{P' > P} \text{Ker}(\text{Spec}_{P'})
$$
Note that when $P$ is the coarsest $I$-partition with a single part of every color $i \in I$ (which is maximal with respect to the partial order), we have $\CS^+_P = \CS^+$.

\medskip

\begin{claim}
\label{claim}

For any $I$-partition $P$ and any $E \in \CS^+_P$, there exists $E' \in \bar{\CS}^+ \cap \CS^+_P$ such that
\begin{equation}
\label{eqn:claim}
\emph{Spec}_P(E) = \emph{Spec}_P(E')
\end{equation}

\end{claim}

\medskip 

\noindent By descending induction with respect to the partial order on $I$-partitions, Claim \ref{claim} allows us to recursively associate to any $E \in \CS^+$ an element $E' \in \bar{\CS}^+$ such that \eqref{eqn:claim} holds for all $I$-partitions $P$. When $P$ is the finest $I$-partition with all parts equal to 1 (which is minimal with respect to the partial order), we have that $\text{Spec}_P = \text{Id}$, which implies that $E = E'$. This proves the opposite inclusion to \eqref{eqn:star} and thus completes the proof of Theorem \ref{thm:shuffle int}.

\medskip

\noindent It remains to prove Claim \ref{claim}. Let us consider the $I$-partition
\begin{equation}
\label{eqn:transpose}
m_i^{(a)} = \left| \{b \geq 1 \text{ s.t. } n_i^{(b)} \geq a \}\right|
\end{equation}
for all $i \in I$ and $a \geq 1$ (we allow $m_i^{(a)} = 0$ for large enough $a$). Define 
$$
\bm^{(a)} = (m_i^{(a)})_{i \in I} \in \nn
$$
for all $a \geq 1$, and denote the partial sums by
$$
\mu_i^{(a)} = m_i^{(1)}+\dots+m_i^{(a)}
$$
for all $i,a$. With this in mind, let us consider the shuffle element
\begin{multline}
\label{eqn:e prime}
E' = \text{pSym} \left[  \prod_{a \geq 1} e_{\bm^{(a)},g_a}(z_{i,\mu_i^{(a-1)}+1},\dots,z_{i,\mu_i^{(a)}})_{i \in I} \right. \\ \left. \prod_{i,j \in I} \prod_{a<b}  \prod_{r=\mu_i^{(a-1)}+1}^{\mu_i^{(a)}} \prod_{s=\mu_j^{(b-1)}+1}^{\mu_j^{(b)}} \zeta_{ij} \left(\frac {z_{ir}}{z_{js}} \right) \right]
\end{multline}
for certain sums of products of to-be-defined symmetric Laurent polynomials 
\begin{equation}
\label{eqn:tensor}
g_1(z_{i,\mu_i^{(0)}+1},\dots,z_{i,\mu_i^{(1)}})_{i \in I} \otimes g_2(z_{i,\mu_i^{(1)}+1},\dots,z_{i,\mu_i^{(2)}})_{i \in I} \otimes \dots
\end{equation}
The notation ``pSym''  in \eqref{eqn:e prime} refers to the operator of partial symmetrization which takes a function that is already symmetric in 
\begin{equation}
\label{eqn:variables}
z_{i,\mu_i^{(a-1)}+1},\dots,z_{i,\mu_i^{(a)}}
\end{equation}
for all $i\in I$, $a \in \{1,\dots,d_i\}$, and produces a function symmetric in $z_{i1},\dots,z_{in_i}$ for all $i \in I$. We will now show that the tensor \eqref{eqn:tensor} can be chosen so that the shuffle element $E'$ satisfies the conditions of Claim \ref{claim}. Firstly, $E'$ is a sum of shuffle products of $e_{\bm,g}$'s by construction, so it lies in $\bar{\CS}^+$. Secondly, we must show that 
\begin{equation}
\label{eqn:secondly}
\text{Spec}_{P'}(E') = 0
\end{equation}
for all $I$-partitions $P' > P$. To this end, we must specialize the variables of \eqref{eqn:e prime} to the geometric progressions 
\begin{equation}
\label{eqn:progression}
\Big\{z_{i1},\dots,z_{in_i}\Big\}  = \left\{x_{i1},x_{i1}q,\dots,x_{i1}q^{{n'}_i^{(1)}-1}, \dots, x_{id'_i},x_{id'_i}q,\dots,x_{id'_i} q^{{n'}_i^{(d'_i)}-1} \right\}
\end{equation}
where we write $P' = \{{n'}_i^{(1)} \geq \dots \geq {n'}_i^{(d_i')}>0\}_{i \in I}$. The fact that $P'>P$ implies that there must exist some $i \in I$ and $b'\geq 1$ such that
\begin{equation}
\label{eqn:such that}
{n'}_i^{(1)} = n_i^{(1)}, \dots, {n'}_i^{(b'-1)} = n_i^{(b'-1)} \text{ and } {n'}_i^{(b')} > n_i^{(b')}
\end{equation}
However, let us see what it would take to get a non-zero result when we plug the geometric progressions \eqref{eqn:progression} into the right-hand side of \eqref{eqn:e prime}. Because any $e_{\bm,g}$ vanishes when any two of its variables of the same color differ by a ratio of $q$, the only way in which we could obtain a non-zero result is if the variables
$$
x_{ib},x_{ib}q,\dots,x_{ib}q^{{n'}_i^{(b)}-1} \text{ were plugged into distinct factors } e_{\bm^{(a)},g_a}
$$
for all $b \geq 1$. This is only possible if for all $i \in I$, $b\geq 1$, we have
$$
{n'}_i^{(1)} + \dots + {n'}_i^{(b)} \leq \left|\{a\geq 1 \text{ s.t. }m_i^{(a)} \geq 1 \}\right| + \dots + \left|\{a\geq 1 \text{ s.t. }m_i^{(a)} \geq b \}\right|
$$
Since the right-hand side of the expression above is equal to ${n}_i^{(1)} + \dots + {n}_i^{(b)}$ by \eqref{eqn:transpose}, we obtain a contradiction of \eqref{eqn:such that}, so there is no choice but for \eqref{eqn:secondly} to be true. Thirdly and lastly, we must ensure that
\begin{equation}
\label{eqn:thirdy}
\text{Spec}_P(E') = \text{Spec}_P(E)
\end{equation}
and for this we will use our freedom in choosing the tensor \eqref{eqn:tensor}. Due to the discussion in our proof of \eqref{eqn:secondly} above, there is only one way to plug the specialization 
$$
\Big\{z_{i1},\dots,z_{in_i}\Big\}  = \left\{x_{i1},x_{i1}q,\dots,x_{i1}q^{{n}_i^{(1)}-1}, \dots, x_{id_i},x_{id_i}q,\dots,x_{id_i} q^{{n}_i^{(d_i)}-1} \right\}
$$
into the right-hand side of \eqref{eqn:e prime} and get a non-zero result. Namely, we must plug 
\begin{align*}
&x_{i1},x_{i2},x_{i3},\dots, \quad \ \text{ into the variables of } e_{\bm^{(1)},g_1} \\
&x_{i1}q,x_{i2}q,x_{i3}q,\dots, \text{ into the variables of } e_{\bm^{(2)},g_2}
\end{align*}
etc (the reason why variables with smaller powers of $q$ must be plugged into $e_{\bm^{(a)},g_a}$ with smaller $a$ is the presence of the zeta factors in \eqref{eqn:e prime} and the fact that $\zeta_{ii}(q)=0$ for all $i\in I$). When this is done, we see that
\begin{multline}
\label{eqn:factor 1}
\text{Spec}_P(E') = g_1(x_{i1},x_{i2},\dots)_{i \in I} g_2(x_{i1}q,x_{i2}q,\dots)_{i \in I} \dots \\ \prod_{i \in I} \prod_{a , b} \left(1 - \frac {x_{ia}}{x_{ib}q}\right)^{\min(n_i^{(a)}, n_i^{(b)})} \prod_{i,j \in I} \prod_{a,b} \mathop{\prod^{r < s}_{0 \leq r < n_i^{(a)}}}_{0 \leq s < n_j^{(b)}} \zeta_{ij} \left(\frac {x_{ia}q^{r}}{x_{jb}q^{s}}\right)
\end{multline}
We will rearrange the product of factors in the expression above as
\begin{equation}
\label{eqn:factor 2}
\text{Spec}_P(E') = g_1(x_{i1},x_{i2},\dots)_{i \in I} g_2(x_{i1}q,x_{i2}q,\dots)_{i \in I} \dots \times \Pi_1 \Pi_2 \Pi_3
\end{equation}
where 
\begin{align}
&\Pi_1 = \prod_{i \in I} \prod_{1 \leq a \leq d_i} \left[(1-q^{-1})(1-q^{-2})\dots(1-q^{-n_i^{(a)}}) \right] \label{eqn:pi 1} \\
&\Pi_2 = \prod_{\alpha : i \rightarrow j} \mathop{\prod_{1 \leq a \leq d_i}}_{1 \leq b \leq d_j} \underbrace{\mathop{\prod^{r < s}_{0 \leq r < n_i^{(a)}}}_{0 \leq s < n_j^{(b)}} \left(1 - \frac {x_{ia} q^r t_{\alpha}}{x_{jb} q^s} \right)  \mathop{\prod^{r > s}_{0 \leq r < n_i^{(a)}}}_{0 \leq s < n_j^{(b)}} \left(1 - \frac {x_{ia} q^{r} t_\alpha}{x_{jb} q^{s+1}} \right)}_{= \prod_{c \in \BZ} \left(1 - \frac {x_{ia} t_{\alpha}}{x_{jb} q^c} \right)^{\chi_{n_i^{(a)},n_j^{(b)}}(c)} \text{ by \eqref{eqn:explicit chi}}} \label{eqn:pi 2} \\
&\Pi_3 = \prod_{i \in I} \prod_{a \neq b} \prod_{s=1+\max(0,n_i^{(a)}-n_i^{(b)})}^{n_i^{(a)}} \left(1 - \frac {x_{ib}}{x_{ia}q^s}\right) \label{eqn:pi 3}
\end{align}
As $E \in \CS^+$, Definition \ref{def:shuffle int} states that
\begin{equation}
\label{eqn:factor 3}
 \Pi_1 \Pi_2 \quad \text{divides} \quad \text{Spec}_P(E)
\end{equation}
However, we will now show that the assumption $\text{Spec}_{P'}(E) = 0$ for all $P' > P$ imposes additional constraints on $\text{Spec}_P(E)$. Specifically, for any $i \in I$, $a \neq b$ and $1+\max(1,n_i^{(a)}-n_i^{(b)}) \leq s \leq n_i^{(a)}$, consider the geometric progressions
\begin{align*}
x_{ia},x_{ia}q,\dots,&x_{ia}q^s,\dots\dots,x_{ia}q^{n_i^{(a)}-1} \\ 
&x_{ib},x_{ib}q,\dots,x_{ib}q^{n_i^{(a)}-s-1},\dots,x_{ib}q^{n_i^{(b)}-1} 
\end{align*}
Further specializing $x_{ib}=x_{ia}q^s$ can also be interpreted as the specialization corresponding to the geometric progressions
\begin{align*}
x_{ia},x_{ia}q,\dots,&x_{ia}q^s,\dots\dots,x_{ia}q^{n_i^{(a)}-1},x_{ia}q^{n_i^{(a)}},\dots, x_{ia}q^{n_i^{(b)}+s-1}\\ 
&x_{ib},x_{ib}q,\dots,x_{ib}q^{n_i^{(a)}-s-1} 
\end{align*}
followed by setting $x_{ib}=x_{ia}q^s$. If we let $P'$ be the $I$-partition obtained from $P$ by replacing the parts
$$
n_i^{(a)},n_i^{(b)} \quad \text{by} \quad n_i^{(b)}+s,n_i^{(a)}-s
$$
as above, then $P'>P$. Thus, the assumption that $\text{Spec}_{P'}(E) = 0$ implies that
$$
\text{Spec}_P(E) \Big|_{x_{ib}=x_{ia}q^s} = 0
$$
which implies that $\text{Spec}_P(E)$ is divisible by the factor $1 - \frac {x_{ib}}{x_{ia}q^s}$. Since these factors are all distinct as $a,b,i,s$ vary (because $q$ is not a root of unity), we conclude that
\begin{equation}
\label{eqn:factor 4}
\Pi_3 \quad \text{divides} \quad \text{Spec}_P(E)
\end{equation}
Since the products $\Pi_1\Pi_2$ and $\Pi_3$ have no linear factors in common (this uses the geometric assumption \eqref{eqn:assumption geometric}), then we conclude that there exists a Laurent polynomial $G$ such that 
\begin{equation}
\label{eqn:factor 5}
\text{Spec}_P(E) = G(x_{i1},\dots,x_{id_i})_{i \in I} \times \Pi_1 \Pi_2 \Pi_3
\end{equation}
Moreover, $G$ is symmetric in $x_{ia}$ and $x_{ib}$ if $n_i^{(a)} = n_i^{(b)}$, due to the color-symmetry of $E$ and the inherent symmetry of the specialization at the $I$-partition $P$. We wish to show that the left-hand sides of \eqref{eqn:factor 2} and \eqref{eqn:factor 5} coincide, which is equivalent to the fact that their right-hand sides coincide:
$$
g_1(x_{i1},x_{i2},\dots)_{i \in I} g_2(x_{i1}q,x_{i2}q,\dots)_{i \in I} \dots = G(x_{i1},\dots,x_{id_i})_{i \in I}
$$
Recall that we have the freedom of choosing the tensor \eqref{eqn:tensor}, but only inasmuch as each $g_a$ is color-symmetric in the variables \eqref{eqn:variables}. Since this symmetry only involves the variables of any given color $i 
\in I$ separately, it suffices to prove the following $|I| = 1$ statement: for any partition $n_1 \geq \dots \geq n_d > 0$ with transposed partition $m_1 \geq m_2 \geq \dots$ and any Laurent polynomial $G(x_1,\dots,x_d)$ which is symmetric in $x_a$ and $x_b$ whenever $n_a = n_b$, we can write
$$
g_1(x_1,\dots,x_{m_1}) g_2(x_1q,\dots,x_{m_2}q,\dots) \dots  = G(x_1,\dots,x_d)
$$
where the $g_a$'s are symmetric Laurent polynomials. This claim was proved in the final part of the proof of Proposition 3.9 of \cite{Integral}, following equation (3.21) therein.

\end{proof}

\section{Appendix B: $K$-theoretic Hall algebras}
\label{sec:b}

\medskip

\subsection{Moduli of quiver representations}
\label{sub:moduli}

For any quiver $Q$ with vertex set $I$ and arrow set $\edge$, we consider 

\medskip 

\begin{itemize}

\item the tripled quiver $\tQ$ as in Subsection \ref{sub:k-ha}

\medskip 

\item the quiver $Q^+$ obtained from $Q$ by adding loops $\omega_i$ at all vertices $i \in I$
    
\end{itemize}

\medskip 

\noindent For all $\bn = (n_i)_{i \in I} \in \nn$, we will consider the affine spaces of $\bn$-dimensional representations of the quivers above, depicted graphically as
\begin{equation}
\label{eqn:representation tilde}
\text{Rep}_{\bn}(\tQ) = \Big\{\underset{\omega_i}{\underset{\circlearrowleft}{\BC}}^{n_i} \xleftrightharpoons[\alpha]{\dalpha} \underset{\omega_j}{\underset{\circlearrowleft}{\BC}}^{n_j}\Big\}
\end{equation}
and 
\begin{equation}
\label{eqn:representation plus}
\text{Rep}_{\bn}(Q^+) = \Big\{\underset{\omega_i}{\underset{\circlearrowleft}{\BC}}^{n_i} \xrightarrow{\alpha} \underset{\omega_j}{\underset{\circlearrowleft}{\BC}}^{n_j}\Big\}
\end{equation}
The affine spaces above are acted on by a torus $T$ as in Subsection \ref{sub:k-ha}: explicitly, $T$ acts on the maps $\alpha,\dalpha,\omega_i$ via the characters $t_\alpha^{-1}, t_{\dalpha}^{-1},q$, respectively, see \eqref{eqn:parameters}-\eqref{eqn:opposite parameters}. We will also consider the closed subsets
\begin{align*}
&\text{Rep}^{\omega-\text{nilp}}_{\bn}(\tQ) \hookrightarrow \text{Rep}_{\bn}(\tQ) \\
&\text{Rep}^{\omega-\text{nilp}}_{\bn}(Q^+) \hookrightarrow \text{Rep}_{\bn}(Q^+)
\end{align*}
corresponding to the condition that all the loops $\omega_i$ act by nilpotent operators. We also have the conjugation action of
$$
\GL_{\bn} = \prod_{i \in I} \GL_{n_i}(\BC)
$$ on all the representation spaces above, which allows us to define the quotient stacks
$$
\mathfrak{M}_{\bn}(\tQ) = \mathrm{Rep}_{\bn}(\tQ)/\GL_{\bn}
$$
and analogously for $Q^+$ or ``$\omega$-nilp'' superscripts. The action of the torus $T$ descends to the stacks above, and the trace of the canonical cubic potential 
$$
\tW = \sum_{\alpha : i \rightarrow j} \Big( \dalpha \alpha \omega_i - \omega_j \alpha \dalpha \Big)
$$
yields a $T$-invariant regular function on all these stacks.

\medskip 

\subsection{Matrix Factorizations} 
\label{MF}

For any triangulated category $\mathcal{C}$, let $K_0(\mathcal{C})$ denote its Grothendieck group. Given a quotient stack $\mathfrak{X}=X/G$ where $X$ is a quasiprojective scheme endowed with an action of a reductive group $G$, we let $K_0(\mathfrak{X}):= K_0(\Perf(\fX))$ be the Grothendieck group of its category of perfect complexes and $G_0(\mathfrak{X}) := K_0(\Db(\Coh(\fX)))$ be the Grothendieck group of its derived category of coherent sheaves. Similarly, $K_0^H$ and $G_0^H$ will refer to the analogous $H$-equivariant $K$-theory groups, given $H \curvearrowright \mathfrak{X}$. With the notations of Subsection \ref{sub:moduli}, we let
$$
\mathrm{MF}^{\mathrm{gr}}(\mathfrak{M}_{\bn}(\tilde{Q}), \Tr(\tilde{W}))
$$
be the category of graded matrix factorizations, in the sense of \cite{padurariu2023categoricalktheoreticdonaldsonthomastheory}. In particular, when $\tW$ is replaced by $0$, then 
$$
\mathrm{MF}^{\mathrm{gr}}(\mathfrak{M}_{\bn}(\tilde{Q}), 0) \simeq \Db(\Coh(\mathfrak{M}_{\bn}(\tilde{Q})))
$$
by \cite[Remark 2.3.7]{toda2021categoricaldonaldsonthomastheorylocal}. Note that one can also consider the usual category of matrix factorizations $\mathrm{MF}(\mathfrak{M}_{\bn}(\tilde{Q}), \Tr(\tilde{W}))$; however by \cite[Corollary 3.13]{toda2021categoricaldonaldsonthomastheorylocal}, the Grothendieck groups of the two categories are isomorphic and we shall identify them in what follows. Let
$$
\nilpstack_{\bn} = \mathrm{Rep}^{\omega-\text{nilp}}_{\bn}(\tQ)/\GL_{\bn} \subset \mathfrak{M}_{\bn}(\tilde{Q})
$$
denote the closed substack where all the loops $\omega_i$ act by nilpotent operators. Then we consider the subcategory (see \cite[Equation 2.2.4]{toda2021categoricaldonaldsonthomastheorylocal} for a general definition)
\begin{equation}  \label{eqn:subcategory}
\mathrm{MF}^{\mathrm{gr}}(\mathfrak{M}_{\bn}(\tilde{Q}), \Tr(\tilde{W}))_{\nilpstack_{\bn}}  \subset \mathrm{MF}^{\mathrm{gr}}(\mathfrak{M}_{\bn}(\tilde{Q}), \Tr(\tilde{W})) \end{equation} defined as the kernel of the restriction functor \[ \mathrm{MF}^{\mathrm{gr}}(\mathfrak{M}_{\bn}(\tilde{Q}), \Tr(\tilde{W})) \rightarrow \mathrm{MF}^{\mathrm{gr}}(\mathfrak{M}_{\bn}(\tilde{Q}) \backslash \nilpstack_{\bn}, \Tr(\tilde{W})).\] 

\medskip

\subsection{Dimensional Reduction}
\label{sub:dimensional reduction}

Let $\mathfrak{X}= X/G$ where $X$ is a smooth affine scheme endowed with an action of a reductive group $G$, and let $E$ be a $G$-equivariant vector bundle on $X$. Suppose there is an action of $\mathbb{C}^*$ on the fibers of $E$ with weight $2$ and let $s \in \Gamma(X,E)$ be a section of $E$. Let $ s^{-1}(0)^{\text{der}} \hookrightarrow X$ be the derived zero locus of $s$ and let $f: \mathrm{Tot}_X(E^{\vee})/G \rightarrow \mathbb{C}$ be the regular function defined by $f(x,v_x) =  \langle s(x),v_x \rangle$ for all $x \in \mathfrak{X}$. In the present context, the phrase dimensional reduction refers to the following equivalence due to \cite{isik2010equivalencederivedcategoryvariety} 
\begin{equation}
\label{eqn:dimensional reduction cat}
\mathrm{MF}^{\mathrm{gr}}(\mathrm{Tot}_X(E^{\vee})/G,f) \simeq \Db(\Coh(s^{-1}(0)^{\text{der}}/G))
\end{equation}
At the level of Grothendieck groups, we obtain the following isomorphism (we may remove the superscript ``der'' in the right-hand side, as the $G$-theory is insensitive to the derived structure, see \cite[Corollary 3.4]{Khan_2022})
\begin{equation}
\label{eqn:dimensional reduction}
K_0(\mathrm{MF}^{\mathrm{gr}}(\mathrm{Tot}_X(E^{\vee})/G,f)) \cong G_0^T(s^{-1}(0)/G).
\end{equation}
Recall the quiver $Q^+$ of Subsection \ref{sub:moduli}, which is obtained from $Q$ by adding loops $\omega_i$ at all the vertices. Let
\begin{equation}
\label{eqn:z subscheme}
Z_{Q^{+},\bn} \subset \mathrm{Rep}_{\bn}(Q^{+})
\end{equation}
be the closed subscheme that parameterizes representations \eqref{eqn:representation plus} for which
\begin{equation}
\label{eqn:equations}
\{ \alpha \omega_i = \omega_j \alpha \}_{\forall \alpha : i \rightarrow j}
\end{equation}
We also write
\begin{equation}
\label{eqn:z derived}
Z^{\text{der}}_{Q^{+},\bn} \subset \mathrm{Rep}_{\bn}(Q^{+})
\end{equation}
for the derived subscheme cut out by the equations \eqref{eqn:equations}. Then taking the equivalence \eqref{eqn:dimensional reduction} at the level of Grothendieck groups implies that (See \cite[3.1.1]{padurariu2022generatorscategoricalhallalgebras} for an application of \eqref{eqn:dimensional reduction cat} to quivers with a cut) and \cite[Corollary 3.4]{Khan_2022})
\begin{equation}
\label{eqn:dim red full}
K^T_0(\mathrm{MF}^{\mathrm{gr}}(\mathfrak{M}_{\bn}(\tilde{Q}), \Tr(\tilde{W})))  \cong G^T_0(Z_{Q^{+},\bn}/\mathrm{GL}_{\bn}).
\end{equation} 
Furthermore, \cite[Proposition 2.3.9]{toda2021categoricaldonaldsonthomastheorylocal} and \cite[Example 2.13]{toda2021categoricaldonaldsonthomastheorylocal} show that 
\begin{equation}
\label{eqn:dim red nilp} 
K^T_0(\mathrm{MF}^{\mathrm{gr}}(\mathfrak{M}_{\bn}(\tilde{Q}), \Tr(\tilde{W})_{\nilpstack_{\bn}})  \cong G^T_0(Z^{\omega\text{-nilp}}_{Q^+,\bn}/\mathrm{GL}_{\bn}) 
\end{equation} 
where 
\begin{equation}
\label{eqn:z nilp subscheme} 
Z^{\omega\text{-nilp}}_{Q^+,\bn} \subset Z_{Q^+,\bn}
\end{equation}
is the closed subscheme of \eqref{eqn:representation plus} determined by the condition that the loops $\omega_i$ all act by nilpotent operators. To provide some more details on \eqref{eqn:dim red nilp}, note that \loccit proves that the equivalence \eqref{eqn:dimensional reduction cat} restricts to an equivalence of categories 
\[ \mathrm{MF}^{\mathrm{gr}}(\mathrm{Tot}_X(E^{\vee})/G,f)_Z \simeq \mathcal{C}_{Z}\] 
for any conical subset $Z/G \subset \mathrm{crit}(f)/G$, where $\mathcal{C}_{Z}$ refers to the subcategory of $\Db(\Coh(s^{-1}(0)^{\mathrm{der}}/\mathrm{G}))$ corresponding to sheaves with singular support in $Z$. By \cite[Example 2.13]{toda2021categoricaldonaldsonthomastheorylocal}, the subcategory $\mathcal{C}_{\nilpstack_{\bn}}$ consists of coherent sheaves on the derived version of $Z^{\omega\text{-nilp}}_{Q^+,\bn}/\mathrm{GL}_{\bn}$. Since the Grothendieck group is insensitive to the derived structure by \cite[Corollary 3.4]{Khan_2022}, we obtain the isomorphism \eqref{eqn:dim red nilp}. 

\medskip

\subsection{The loop nilpotent $K$-theoretic Hall algebra}
\label{sub:loop}

We follow \cite{Padurariu} and subsequent works for basics on $K$-HAs of quivers with potential, itself inspired by the construction of cohomological Hall algebras in \cite{KS}. Take 
\[K_{\tQ,\tW,\bn}^T := K^T_0(\mathrm{MF}^{\mathrm{gr}}(\mathfrak{M}_{\bn}(\tilde{Q}), \Tr(\tilde{W})))\] 
and consider the $K$-theoretic Hall algebra 
\begin{equation}
\label{eqn:k-ha appendix}
K_{\tQ,\tW}^T = \bigoplus_{\bn \in \mathrm{N}^I} K_{\tQ,\tW,\bn}^T
\end{equation}
The multiplication is defined by taking the image in the Grothendieck group of the monoidal structure (\cite{padurariu2021categoricalktheoretichallalgebras}) on the direct sum of matrix factorization categories
\begin{equation}
\label{eqn:catha}
\mathrm{MF}^{\text{gr}}(\mathfrak{M}(\tilde{Q}), \Tr(\tilde{W})):= \bigoplus_{\bn \in \mathbb{N}^I} \mathrm{MF}^{\text{gr}}(\mathfrak{M}_{\bn}(\tilde{Q}), \Tr(\tilde{W})) 
\end{equation}
In a nutshell, the construction of the monoidal structure uses the diagram
\[
\begin{tikzcd}[cramped]
{\mathfrak{M}_{\bn_1}(\tQ) \times \mathfrak{M}_{\bn_2}(\tQ)}
&
{\mathfrak{M}_{\bn_1,\bn_2}(\tQ)}
\arrow[l, "{\pi_1}"']
\arrow[r, "{\pi_2}"]
&
{\mathfrak{M}_{\bn_1+\bn_2}(\tQ)}
\end{tikzcd}
\]
for all $\bn^1,\bn^2 \in \nn$, where $\mathfrak{M}_{\bn_1,\bn_2}(\tQ)$ is the quotient stack of short exact sequences $0 \rightarrow \rho_1 \rightarrow \rho \rightarrow \rho_2 \rightarrow 0$, where $\rho_1$ and $\rho_2$ are of dimension $\bn_1$ and $\bn_2$. The morphism $\pi_1: \mathfrak{M}_{\bn_1,\bn_2}(\tQ) \rightarrow \mathfrak{M}_{\bn_1}(\tQ) \times \mathfrak{M}_{\bn_2}(\tQ)$ which forgets $\rho$ is smooth, while the morphism $\pi_2: \mathfrak{M}_{\bn_1,\bn_2}(\tQ) \rightarrow \mathfrak{M}_{\bn_1+\bn_2}(\tQ)$ which forgets $\rho_1,\rho_2$ is proper. The functorial properties of matrix factorizations allow one to define the pullback $\pi_1^*$ and the pushforward $\pi_{2*}$, and the corresponding composition $\pi_{2*} \circ \pi_1^*$ gives the required monoidal structure on \eqref{eqn:catha}.

\medskip 

\noindent In the present paper, we will consider the loop-nilpotent version of the above $K$-theoretic Hall algebra, defined as follows
\[\khan := K^T_0(\mathrm{MF}^{\mathrm{gr}}(\mathfrak{M}_{\bn}(\tilde{Q}), \Tr(\tilde{W}))_{\nilpstack_{\bn}})\] 
and
\begin{equation}
\label{eqn:kha nilp}
\kha = \bigoplus_{\bn \in \nn} \khan 
\end{equation}
The monoidal structure defined in the previous paragraph induces a monoidal structure
\begin{equation}
\label{eqn:catha nilp}
\mathrm{MF}^{\text{gr}}(\mathfrak{M}(\tilde{Q}), \Tr(\tilde{W}))_{\nilpstack}:= \bigoplus_{\bn \in \mathbb{N}^I} \mathrm{MF}^{\text{gr}}(\mathfrak{M}_{\bn}(\tilde{Q}), \Tr(\tilde{W}))_{\nilpstack_{\bn}} 
\end{equation}
on account of the fact that
$$
\pi_2 \ \text{maps} \ \pi_1^{-1} \left(\nilpstack_{\bn_1} \times \nilpstack_{\bn_2} \right) \ \text{properly onto} \ \nilpstack_{\bn_1+\bn_2}
$$
Thus by \cite[Section 3.5,3.6]{Efimov_2015} (see also \cite[Section 2.2.2]{VV2}), it follows that \eqref{eqn:catha nilp} inherits a monoidal structure from \eqref{eqn:k-ha appendix}, whose image in the Grothendieck group makes \eqref{eqn:kha nilp} an algebra. The associativity of this monoidal structure follows by same argument as in \cite[Theorem 3.3]{Padurariu}. To summarize, we have the following. 

\medskip 

\begin{proposition}
$\ekha$ is an associative algebra, called the loop-nilpotent $K$-HA.
\end{proposition}

\medskip

\begin{proposition}
The map 
\begin{equation} 
\label{eqn:morphismtofullKHA}
\ekha \rightarrow K^T_{\tilde{Q},\tilde{W}} 
\end{equation} 
induced by \eqref{eqn:subcategory} is an algebra homomorphism.

\end{proposition}

\medskip

\begin{proof} It suffices to show that the inclusion \[ \mathrm{MF}^{\mathrm{gr}}(\mathfrak{M}_{\bn}(\tilde{Q}), \Tr(\tilde{W}))_{\nilpstack_{\bn}} \subset \mathrm{MF}^{\mathrm{gr}}(\mathfrak{M}_{\bn}(\tilde{Q}), \Tr(\tilde{W})) \] respects the monoidal structure. However, the functors $\pi_1^*,\pi_{2*}$ involving the category $\mathrm{MF}^{\mathrm{gr}}(\mathfrak{M}_{\bn}(\tilde{Q}), \Tr(\tilde{W}))_{\nilpstack_{\bn}}$ are defined by restriction from $\mathrm{MF}^{\mathrm{gr}}(\mathfrak{M}_{\bn}(\tilde{Q}), \Tr(\tilde{W}))$, and so they commute with the inclusion functor.  
\end{proof}
\begin{proposition} \label{Kdimensionreduction1}
The following diagram commutes \[\begin{tikzcd}
\ekhan & {K^{T}_{\tilde{Q},\tilde{W},\bn}} \\
{G^T_0(Z^{\omega\emph{-nilp}}_{Q^+,\bn}/\GL_{\bn}}) & {G^T_0(Z_{Q^+,\bn}/\GL_{\bn})}
    \arrow[from=1-1, to=1-2]
	\arrow["\cong",from=1-1, to=2-1]
	\arrow["\cong", from=1-2, to=2-2]
	\arrow[from=2-1, to=2-2]
\end{tikzcd}\]
where the vertical maps are given by dimension reduction and the bottom horizontal row is given by direct image. 
\end{proposition}

\medskip

\begin{proof}
By \cite[Proposition 2.39]{toda2021categoricaldonaldsonthomastheorylocal} the following diagram commutes \[\begin{tikzcd}[cramped]
	{\mathrm{MF}^{\mathrm{gr}}(\mathfrak{M}_{\bn}(\tilde{Q}), \Tr(\tilde{W}))_{\nilpstack_{\bn}}} & {\mathrm{MF}^{\mathrm{gr}}(\mathfrak{M}_{\bn}(\tilde{Q}), \Tr(\tilde{W}))} \\
	{\mathcal{C}_{\nilpstack_{\bn}}} & {\Db(\Coh(Z^{\text{der}}_{Q^+,\bn}/\GL_{\bn}))}
	\arrow[hook, from=1-1, to=1-2]
	\arrow[from=1-1, to=2-1]
	\arrow[from=1-2, to=2-2]
	\arrow[hook, from=2-1, to=2-2]
\end{tikzcd}\]
Taking their Grothendieck group implies the claim.

\end{proof}

\noindent We furthermore have an algebra homomorphism
\begin{equation} \label{eqn:Kshuffle morphism}
    {K}^T_{\tilde{Q},\tilde{W}} \rightarrow K^T_{\tilde{Q}}
\end{equation}
by \cite[Proposition 3.6]{padurariu2021categoricalktheoretichallalgebras}, where the right-hand side is defined by replacing the potential $\tW$ by 0 in the definition of the $K$-HA. Similar to the previous result, we have the following.

\medskip 

\begin{proposition} \label{Kdimensionreduction2}
    The following diagram commutes \[\begin{tikzcd}
	{K^T_{\tilde{Q},\tilde{W},\bn} } & {K^T_{\tilde{Q},\bn} } \\
	{G^T_0({Z_{Q^+,\bn}/\GL_{\bn}})} & {G^T_0(\mathrm{Rep}_{\bn}(Q^+)/\GL_{\bn})}
	\arrow[from=1-1, to=1-2]
	\arrow["\cong",from=1-1, to=2-1]
	\arrow["\cong",from=1-2, to=2-2]
	\arrow[from=2-1, to=2-2]
\end{tikzcd}\]
where the vertical maps are given by dimension reduction and the bottom horizontal row is given by direct image.
\end{proposition}

\medskip

\begin{proof}
    This follows from \cite[Lemma 2.4.5]{toda2021categoricaldonaldsonthomastheorylocal}. 
\end{proof}

\noindent Since $\mathrm{Rep}_{\bn}(Q^{+})$ is an affine space, we shall henceforth identify 
\begin{equation}
\label{eqn:identify app}
G^T_{0}(\mathrm{Rep}_{\bn}(Q^{+})) \cong K^T_{\tilde{Q},\bn} \cong \CV^+_{\bn}
\end{equation}
as $K^T\point$ modules. It is proved in \cite[Proposition 3.4]{padurariu2023categoricalktheoreticdonaldsonthomastheory} that the $K$-theoretic Hall algebra structure coincides with the shuffle algebra structure on $\CV^+$. Let 
\begin{equation}
\label{eqn:iota appendix}
\iota \colon \kha \rightarrow K^T_{\tilde{Q}} 
\end{equation}
be the composition of the algebra homomorphisms \eqref{eqn:morphismtofullKHA} and \eqref{eqn:Kshuffle morphism}. We will write 
\begin{equation}
\label{eqn:iota n appendix}
\iota_{\bn}: \khan \rightarrow K^T_{\tilde{Q},\bn}
\end{equation}
for the restriction of \eqref{eqn:iota appendix} to the $\bn$-th graded summand, for all $\bn \in \nn$.  

\medskip 

\begin{lemma}
\label{lem:one}
The image of the map $\iota_{\bn}$ is contained in $\emph{Im }\jmath_{\bn,*}$, where
\begin{equation}
\label{eqn:jmath}
\jmath_{\bn} : Z^{\omega\emph{-nilp}}_{Q^+,\bn} \hookrightarrow \mathrm{Rep}_{\bn}(Q^{+})
\end{equation}
is the composition of the closed embeddings \eqref{eqn:z subscheme} and \eqref{eqn:z nilp subscheme}. We are implicitly using \eqref{eqn:identify app} to identify the codomains of $\iota_{\bn}$ and $\jmath_{\bn,*}$.

\end{lemma}

\medskip

\begin{proof} This follows immediately from Propositions \ref{Kdimensionreduction1} and \ref{Kdimensionreduction2}.  
\end{proof}

\medskip 

\subsection{$K$-HA generators}
\label{sub:kha generators}

Let $\mathcal{O}_{\bn}^{\omega=0}$ denote the structure sheaf of the closed subscheme of $Z^{\omega\text{-nilp}}_{Q^+,\bn}$ defined by $\omega=0$, and let $[\mathcal{O}_{\bn}^{\omega=0}]$ denote its class in $G^T_0(Z^{\omega\text{-nilp}}_{Q^+,\bn})$. 

\medskip

\begin{lemma}
For any $\bn$, we have 
\begin{equation}
\label{eqn:koszul}
\jmath_{\bn,*}([\mathcal{O}_{\bn}^{\omega=0}]) = \prod_{i \in I} \prod_{1 \leq a , b \leq n_i} \left(1 - \frac {z_{ia}}{z_{ib}q}\right)
\end{equation}
\end{lemma}

\medskip

\begin{proof} Under the inclusion $Z^{\omega\text{-nilp}}_{Q^+,\bn} \hookrightarrow \text{Rep}_{\bn}(Q^+)$, the closed subscheme $\{\omega = 0\}$ is mapped to the subscheme of the same name of the affine space $\text{Rep}_{\bn}(Q^+)$. Thus,
$$
\jmath_{\bn,*}([\mathcal{O}_{\bn}^{\omega=0}])
$$
coincides with the exterior algebra of the dual normal bundle to the regular subscheme $\{\omega = 0\} \subset \text{Rep}_{\bn}(Q^+)$. The latter is given by the $T$-characters
$$
\left\{\frac {z_{ib}q}{z_{ia}} \right\}_{i \in I, 1\leq a,b \leq n_i}
$$
and so we obtain precisely the right-hand side of \eqref{eqn:koszul}.

\end{proof}

\noindent Since $\jmath_{\bn,*}$ is a $K^{T \times \mathrm{GL}_{\bn}}\point$ module homomorphism, we have for all $g \in K^{T \times \mathrm{GL}_{\bn}}\point$
\begin{equation}
\label{eqn:varepsilon}
\jmath_{\bn,*} \left(\epsilon_{\bn,g} \right) = e_{\bn,g}. \end{equation}
where we write $\epsilon_{\bn,g} = g \otimes [\mathcal{O}_{\bn}^{\omega=0}]$ and $e_{\bn,g}$ is defined in \eqref{eqn:special}. If we let 
$$
\varepsilon_{\bn,g} \in \kha 
$$
be the preimage of $\epsilon_{\bn,g}$ under the isomorphism \eqref{eqn:dim red nilp}, then we conclude that
\begin{equation}
\label{eqn:conclude}
\iota(\varepsilon_{\bn,g}) = e_{\bn,g}
\end{equation}

\medskip 

\subsection{The stratification of $Z^{\omega\text{-nilp}}_{Q^+,\bn}$.} 

Given any $\bn = (n_i)_{i \in I}\in \nn$, recall that an $I$-partition of $\bn$ is a set of non-negative integers 
$$
n_i =n_i^{(1)}+\cdots +n_i^{(d_i)} \quad \text{such that} \quad n_i^{(1)} \geq \cdots \geq n_i^{(d_i)} >0
$$
for all $i \in I$. We will write $\blambda$ for such an $I$-partition, and recall that we also have the multiplicity notation 
$$
\blambda_i = (1^{\lambda_{i,1}}, 2^{\lambda_{i,2}}, \cdots )
$$
where $\lambda_{i,s}$ denotes the number of times $s$ appears in the $i$-th part of $\blambda$. If $\omega = ( \omega_i)_{i \in I}$ is an $I$-tuple of nilpotent operators, then it is completely determined up to conjugation by its Jordan normal form, which can be encoded by a $I$-partition $\blambda$. With this in mind, we let
\begin{equation}
\label{eqn:stratum}
Z_{\blambda} \subset Z^{\omega\text{-nilp}}_{Q^+,\bn}
\end{equation}
be the closed $T \times \GL_{\bn}$ invariant subscheme where $\omega$ has Jordan normal form $\blambda$.

\medskip 

\begin{proposition}  \label{Ktorusfreeness}

If $q : T \rightarrow \BC^*$ is surjective, then $G^{T \times \mathrm{GL}_{\bn}}_0(Z_{\blambda})$ is a free $K^T\point$ module for any $I$-partition $\blambda$. 
\end{proposition}

\medskip 

\begin{proof} For any $I$-partition $\blambda$, we fix an $I$-tuple of loops $\omega_{\blambda} = (\omega_{\blambda_i})_{i \in I}$ of the given Jordan type. We let 
$$
H_{\blambda}^{\BC^*} = \Big\{(a,g_i)_{i \in I} \text{ s.t. } g_i \omega_{\blambda_i} g_i^{-1} = a \omega_{\blambda_i} \Big\} \subset \BC^* \times \GL_{\bn}
$$
and define $H_{\blambda}^T$ as the subgroup which completes the following commutative diagram
\[\begin{tikzcd}
{H_{\blambda}^T} & {T \times \GL_{\bn}} \\
{H_{\blambda}^{\BC^*}} & {\BC^* \times \GL_{\bn}}
\arrow[from=1-1, to=1-2]
\arrow[from=1-1, to=2-1]
\arrow[ from=1-2, to=2-2]
\arrow[from=2-1, to=2-2]
\end{tikzcd}\]
where the vertical arrow corresponds to the surjective homomorphism $q : T \rightarrow \BC^*$. We have the following isomorphism of $T \times \GL_{\bn}$ schemes
\begin{equation}
\label{eqn:stacks}
Z_{\blambda} \cong (T \times \GL_{\bn}) \underset{H_{\blambda}^T}{\times} \prod_{\alpha: i \rightarrow j} \prod_{n,n^{\prime} \geq 1} \mathrm{Hom}_{\BC[t]\text{-mod}}(\mathbb{C}[t]/t^n, \mathbb{C}[t]/t^{n^{\prime}})^{\lambda_{s(\alpha),n}\lambda_{t(\alpha),n^{\prime}}}
\end{equation}
which is due to the fact that the fiber of $Z_{\blambda}$ over a given $\omega$ of Jordan type $\blambda$ is an affine space (with the direct product of Hom spaces in the right-hand side being precisely the aforementioned affine space over the fixed $\omega_{\blambda}$). Therefore, we have identifications
\begin{equation}
\label{eqn:display}
G^{T \times \mathrm{GL}_{\bn}}_0(Z_{\blambda}) \cong G^{H_{\blambda}^T}_0(\text{affine space}) = K^{H_{\blambda}^T}\point = K^{H_{\blambda}^{\BC^*}}\point \bigotimes_{K^{\BC^*}\point} K^T\point
\end{equation}
Since the projection $H_{\blambda}^{\BC^*} \rightarrow \BC^*$ is clearly surjective, then $K^{H_{\blambda}^{\BC^*}}\point$ is torsion free (and hence free) over the principal ideal domain $K^{\BC^*}\point$. Therefore, the right-hand side of \eqref{eqn:display} is a free $K^T\point$ module, and therefore so is the left-hand side.

\end{proof}

\begin{proposition}  \label{Ktorsionfreenessofjordanform}

Let $T' \subset T$ be the kernel of the character $q$, and let 
$$
A \cong (\BC^*)^I = Z( \GL_{\bn} )
$$
be the torus of scalar matrices. Then $G^{T \times \GL_{\bn}}_0(Z_{\blambda})$ is torsion free over 
\begin{equation}
\label{eqn:s}
S = \text{complement of the kernel of } K^{T \times \GL_{\bn}}\point \rightarrow K^{T' \times A}\point 
\end{equation}

\end{proposition} 

\medskip

\begin{proof} Because $G^{T \times \GL_{\bn}}_0(Z_{\blambda}) \cong K^{H_{\blambda}^T}\point$ by \eqref{eqn:display}, the claim follows from the fact that the ring homomorphism that appears in the right-hand side of \eqref{eqn:s} factors as
$$
K^{T \times \GL_{\bn}}\point \rightarrow K^{H_{\blambda}^T}\point \rightarrow K^{T' \times A}\point 
$$
In turn, this is because $T' \times A$ is contained in $H_{\blambda}^T$, which can be seen from the fact that $T' = \text{Ker }q$ and conjugation by scalar matrices leave the loops $\omega$ unchanged. 

\end{proof}

\subsection{Proof of injectivity}

We will use the stratification
\begin{equation}
\label{eqn:stratification}
Z^{\omega\text{-nilp}}_{Q^+,\bn} = \bigsqcup_{\blambda \text{ is an } I\text{-partition of }\bn} Z_{\blambda}
\end{equation}
to prove Theorem \ref{thm:injective}. 

\medskip 

\begin{proposition} \label{Ktorsionfreenessmultiplicativesubset} If $q : T \rightarrow \BC^*$ is surjective, then 
$$
G^{T \times \GL_{\bn}}_0(Z^{\omega\emph{-nilp}}_{Q^+,\bn})
$$
is free over the ring $K^{T}\point$, and torsion-free over the multiplicative set $S$ of \eqref{eqn:s}. 

\end{proposition}

\medskip 

\begin{proof} Because of \eqref{eqn:stratification}, $G^{T \times \GL_{\bn}}_0(Z^{\omega\text{-nilp}}_{Q^+,\bn})$ is a successive extension of 
$$
\Big\{ G_0^{T \times GL_{\bn}}(Z_{\blambda}) \Big\}_{\blambda \text{ is an } I\text{-partition of }\bn}
$$
(more precisely, we apply Lemma 2.4.3 of \cite{VV} to reduce the statement above to the fact that the equivariant topological $K$-theory of $Z^{\omega\text{-nilp}}_{Q^+,\bn}$ is a successive extension of the equivariant topological $K$-theory groups of $Z_{\blambda}$; because of \eqref{eqn:stacks}, the latter has vanishing $K_1$, due to being homotopic to the classifying space of a complex algebraic group). The properties of being free over the ring $K^{T}\point$ and torsion-free over the multiplicative set $S$ are preserved under taking successive extensions of $K^{T \times \GL_{\bn}}\point$ modules, so Propositions \ref{Ktorusfreeness} and \ref{Ktorsionfreenessofjordanform} imply the required conclusion.

\end{proof}

\medskip 

\noindent Recall the geometric assumption \eqref{eqn:assumption geometric}, which implies both the fact that $q : T \rightarrow \BC^*$ is surjective, and the fact that the restriction of
$$
t_{\alpha} : T \rightarrow \BC^*
$$
to $T' = \text{Ker }q$ is non-trivial, for all $\alpha \in \Omega_{\text{loop}}$. Recall that we wish to prove Theorem \ref{thm:injective}, which states that the map
$$
\iota_{\bn}: \khan \rightarrow K^T_{\tilde{Q},\bn}
$$
is injective. To this end, we will use the following version of the Thomason localization theorem in algebraic $K$-theory (\cite{Thomason}): if a reductive group $H$ acts on a finite type separated scheme $X$, then for any subtorus $N \subseteq Z(H)$ and any $H$-invariant subscheme $Y$ that contains the fixed point locus
\begin{equation}
\label{eqn:three}
X^{N} \subseteq Y \subseteq X 
\end{equation}
the direct image map induces an isomorphism
\begin{equation}
\label{eqn:localization}
G_0^H(Y)[S^{-1}] \cong G_0^H(X)[S^{-1}]
\end{equation}
where $S$ is the complement of the kernel of the restriction $K^H\point \rightarrow K^{N}\point$. Indeed, the usual localization isomorphism is stated with the following modifications

\medskip 

\begin{itemize}[leftmargin=*]

\item $Y = X^N$; however, we have $Y^{N} = X^{N}$ due to the assumption \eqref{eqn:three}, and so both sides of \eqref{eqn:localization} are isomorphic to $G_0^H(X^N)[S^{-1}]$;

\medskip

\item $H$ is a torus; however, the $H$-equivariant $G$-theory of $X$ is identified with the Weyl group invariants of its maximal torus equivariant $G$-theory, and the fact that $N \subseteq Z(H)$ implies that the maximal torus contains $N$. Formula \eqref{eqn:localization} follows by taking Weyl invariants in the corresponding statement for the maximal torus.

\end{itemize} 

\medskip 

\begin{proof} \emph{of Theorem \ref{thm:injective}:} Recall that $\mathrm{Rep}_{\bn}(Q^+)$ is an affine space. By Propositions \ref{Kdimensionreduction1} and \ref{Kdimensionreduction2}, it suffices to prove the injectivity of the map 
\begin{equation}
\label{eqn:inj 0}
G^{T \times \mathrm{GL}_{\bn}}_0(Z^{\omega\text{-nilp}}_{Q^+,\bn}) \xrightarrow{\jmath_{\bn,*}} G^{{T \times \mathrm{GL}_{\bn}}}_0(\mathrm{Rep}_{\bn}(Q^+)) = K^{T \times \GL_{\bn}}\point
\end{equation}
induced by the closed embedding \eqref{eqn:jmath}. Because the character $q : T \rightarrow \BC^*$ is surjective, any quiver representation \eqref{eqn:representation plus} that is $T$-invariant must have the $\omega_i$'s nilpotent. Therefore, we may invoke the localization theorem \eqref{eqn:localization} for $T \subset T \times \GL_{\bn} \curvearrowright Z_{Q^+,\bn}$ to conclude that we have an isomorphism of $K^{T \times \mathrm{GL}_{\bn}}\point$ modules
\begin{equation}
\label{eqn:inj 1}
G^{T \times \mathrm{GL}_{\bn}}_0(Z^{\omega\text{-nilp}}_{Q^+,\bn}) \bigotimes_{K^T\point} \mathrm{Frac}(K^T\point) \cong G^{T \times \mathrm{GL}_{\bn}}_0(Z_{Q^+,\bn}) \bigotimes_{K^T\point} \mathrm{Frac}(K^T\point)
\end{equation}
We now apply the localization theorem to $T' \times A \subset T \times \GL_{\bn} \curvearrowright \text{Rep}_{\bn}(Q^+)$. Since any fixed point of this action must lie in the affine space $\{\alpha = 0\}_{\alpha \in \edge}$ due to the geometric assumption that $t_{\alpha}|_{T'}$ is non-trivial for any loop $\alpha$, then we have an isomorphism
\begin{equation}
\label{eqn:inj 2}
G^{T \times \mathrm{GL}_{\bn}}_0(Z_{Q^+,\bn}) [S^{-1}] \cong G^{T \times \mathrm{GL}_{\bn}}_0(\text{Rep}_{\bn}(Q^+)) [S^{-1}] = K^{T \times \mathrm{GL}_{\bn}} \point [S^{-1}]
\end{equation}
where $S$ is the complement of the kernel of $K^{T \times \GL_{\bn}}\point \rightarrow K^{T' \times A}\point$. Combining the isomorphisms \eqref{eqn:inj 1} and \eqref{eqn:inj 2} gives us an isomorphism
\begin{equation}
\label{eqn:inj 3}
G^{T \times \mathrm{GL}_{\bn}}_0(Z^{\omega\text{-nilp}}_{Q^+,\bn}) \bigotimes_{K^T\point} \mathrm{Frac}(K^T\point)[S^{-1}] \cong K^{\mathrm{GL}_{\bn}}\point \otimes \mathrm{Frac}(K^T\point)[S^{-1}]
\end{equation}
which is induced by direct image. However, Proposition \ref{Ktorsionfreenessmultiplicativesubset} implies that we have an injection of $K^{{T \times \mathrm{GL}_{\bn}}}\point$ modules
\begin{equation}
\label{eqn:inj 4}
G^{T \times \mathrm{GL}_{\bn}}_0(Z^{\omega\text{-nilp}}_{Q^+,\bn}) \hookrightarrow G^{T \times \mathrm{GL}_{\bn}}_0(Z^{\omega\text{-nilp}}_{Q^+,\bn}) \bigotimes_{K^T\point} \mathrm{Frac}(K^T\point)[S^{-1}]
\end{equation}
Combining \eqref{eqn:inj 3} and \eqref{eqn:inj 4} implies that the direct image embeds 
$$
G^{T \times \mathrm{GL}_{\bn}}_0(Z^{\omega\text{-nilp}}_{Q^+,\bn}) \hookrightarrow G^{T \times \mathrm{GL}_{\bn}}_0(\text{Rep}_{\bn}(Q^+)) = K^{\mathrm{GL}_{\bn}}\point
$$
which is exactly what we needed to prove.

\end{proof}

\subsection{Yu Zhao's trick}
\label{sub:yu}

We will prove Proposition \ref{prop:generator} by adapting the ideas of \cite{Zhao} to the situation at hand. 

\medskip 

\begin{proof} \emph{of Proposition \ref{prop:generator}:} As a consequence of Proposition \ref{thm:injective}, we have an injection 
$$
\iota : \kha \hookrightarrow \kzero \cong \CV^+
$$
so it suffices to show that the image of this injection is generated by the images of the $\varepsilon_{\bn,g}$. However, by \eqref{eqn:conclude} the latter images are the $e_{\bn,g}$'s, so we have
$$
\iota(\kha) \supseteq \bar{\CS}^+
$$
(with the notation as in \eqref{eqn:subalgebra generated by}). As shown in Theorem \ref{thm:shuffle int}, we have
$$
\bar{\CS}^+ = \CS^+
$$
and thus Proposition \ref{prop:generator} would follow once we prove the inclusion
$$
\iota(\kha) \subseteq \CS^+
$$
Since $\CS^+$ is the set of color-symmetric Laurent polynomials which satisfy the divisibility conditions of Definition \ref{def:shuffle int}, we must show that any $E \in \text{Im }\iota_{\bn}$ has the property that for any $I$-composition $P$, 
\begin{equation}
\label{eqn:specialization final}
G = E \left(x_{i1},x_{i1}q,\dots,x_{i1}q^{n_i^{(1)}-1}, \dots, x_{id_i},x_{id_i}q,\dots,x_{id_i} q^{n_i^{(d_i)}-1} \right)_{i \in I}
\end{equation}
is divisible by
\begin{align}
&\Pi_1 = \prod_{i \in I} \prod_{a=1}^{d_i} \left[ (1-q^{-1})(1-q^{-2})\dots(1-q^{-n_i^{(a)}}) \right] \\
&\Pi'_2 = \prod_{\alpha : i \rightarrow j} \mathop{\prod^{r<s}_{1 \leq a \leq d_i}}_{1 \leq b \leq d_j} \mathop{\prod_{0 \leq r < n_i^{(a)}}}_{0 \leq s < n_j^{(b)}} \left(1- \frac {x_{ia}t_{\alpha}q^r}{x_{jb}q^{s}} \right) \\
&\Pi_2'' = \prod_{\alpha : i \rightarrow j} \mathop{\prod^{r>s}_{1 \leq a \leq d_i}}_{1 \leq b \leq d_j} \mathop{\prod_{0 \leq r < n_i^{(a)}}}_{0 \leq s < n_j^{(b)}} \left(1- \frac {x_{ia}t_{\alpha}q^r}{x_{jb}q^{s+1}} \right) 
\end{align}
(note that $\Pi_2'\Pi_2'' = \Pi_2$ of \eqref{eqn:pi 2}). In what follows, we will prove that $\Pi_1\Pi_2' | G$ and $\Pi_1\Pi_2'' | G$. As an immediate consequence of (and motivation for) Assumption \soft of \eqref{eqn:assumption soft}, we would therefore conclude the required divisibility $\Pi_1\Pi_2'\Pi_2'' | G$. Therefore, let us consider the locally closed subset
$$
\upsilon : U \hookrightarrow \Big\{\underset{\omega_i}{\underset{\circlearrowleft}{\BC}}^{n_i} \xrightarrow{\alpha} \underset{\omega_j}{\underset{\circlearrowleft}{\BC}}^{n_j}\Big\}
$$
parameterizing maps $\omega_i$ and $\alpha$ such that 

\medskip 

\begin{itemize}[leftmargin=*]

\item $\omega_i$ are block diagonal with respect to the composition $n_i = n_i^{(1)} + \dots + n_i^{(d_i)}$, and the blocks take the form
\begin{equation}
\label{eqn:matrix}
\begin{pmatrix} \beta & 1 & 0 & \dots & 0 & 0 & 0 \\ \gamma & 0 & 1 & \dots & 0 & 0 & 0 \\ \delta & 0 & 0 & \dots & 0 & 0 & 0 \\ \vdots & \vdots & \vdots & \ddots & \vdots & \vdots & \vdots \\ o & 0 & 0 & \dots & 0 & 1 & 0 \\ \pi & 0 & 0 & \dots & 0 & 0 & 1 \\ \rho & 0 & 0 & \dots & 0 & 0 & 0 \end{pmatrix}
\end{equation}

\medskip 

\item $\alpha : \BC^{n_i} \rightarrow \BC^{n_j}$ takes the form
\begin{equation}
\label{eqn:alpha in coordinates}
\mathop{\sum_{1 \leq a \leq d_i}}_{1 \leq b \leq d_j} \mathop{\sum^{r<s}_{0 \leq r < n_i^{(a)}}}_{0 \leq s < n_j^{(b)}} c_{i,a,r}^{j,b,s} M^{i,a,r}_{j,b,s}
\end{equation}
where $M^{i,a,r}_{j,b,s} : \BC^{n_i} \rightarrow \BC^{n_j}$ is the linear map which takes the $r$-th standard basis vector in the $a$-th diagonal block of the domain to the $s$-th standard basis vector in the $b$-th diagonal block of the codomain (with respect to the blocks \eqref{eqn:matrix}), and all other basis vectors to 0.

\medskip 

\item the coordinates $\beta,\gamma,\delta,\dots,o,\pi,\rho, c_{i,a,r}^{j,b,s} \in \BC$ are not all 0, so
$$
U \cong \BA^N \backslash 0 
$$
where $N$ is the total number of such coordinates.

\end{itemize}

\medskip 

\noindent In order for a matrix $\omega_i$ with blocks \eqref{eqn:matrix} to be nilpotent, we must have $\beta = \gamma = \delta = \dots = o = \pi = \rho = 0$. Furthermore, to satisfy the equation $\omega_j \alpha = \alpha \omega_i$ for all arrows $\alpha : i \rightarrow j$, all the coordinates $c_{i,a,r}^{j,b,s}$ must be 0. We conclude that 
$$
(\text{Im }\jmath_{\bn}) \cap U = \varnothing 
$$
and so Lemma \ref{lem:one} implies that $\upsilon^*(E) = 0$. However, the fact that we fixed the super-diagonal entries of the $\omega_i$'s to be equal to 1 is equivalent to the fact that the elementary maximal torus characters of $\GL_{\bn}$ have been specialized to
\begin{equation}
\label{eqn:variables final}
z_{i,\nu_i^{(a-1)}+1}q = z_{i,\nu_i^{(a-1)}},\dots,z_{i,\nu_i^{(a)}} q = z_{i,\nu_i^{(a)}-1}
\end{equation}
where $\nu_i^{(a)} = n_i^{(1)} + \dots + n_i^{(a)}$ for all $i \in I$ and $a \in \{1,\dots,d_i\}$. This is manifestly the same condition as replacing $E$ by the specialization $G$ of \eqref{eqn:specialization final}. With respect to the specialization \eqref{eqn:variables final}, the coordinates of $U$ are acted on by the torus characters
\begin{align*}
&q,q^2,\dots,q^{n_i^{(a)}} \ \ \text{for } \beta,\gamma,\delta,\dots,o,\pi,\rho \text{ in \eqref{eqn:matrix}} \\
&\frac {x_{jb}q^{s}}{x_{ia} t_\alpha q^r} \quad \qquad \text{ for } c_{i,a,r}^{j,b,s} \text{ in \eqref{eqn:alpha in coordinates}}
\end{align*}
It is easy to see that an element in the equivariant $K$-theory of $U\cong \BA^N \backslash 0$ vanishes if and only if it is divisible by the exterior power of the vector space with the above torus characters, which is none other than $\Pi_1\Pi_2'$. One proves that $G$ is divisible by $\Pi_1\Pi_2''$ analogously, by swapping the roles of the arrows $\alpha$ and $\dalpha$ in the preceding argument.

\end{proof}

\begin{remark}

The proof above establishes the fact that elements in the loop-nilpotent $K$-HA satisfy the divisibility conditions of Definition \ref{def:shuffle int}. We also note the parallel result of \cite[Lemma 4.1]{PadurariuTodaCat2}, who constructed elements in the $K$-HA of the tripled Jordan quiver supported on (what is essentially) the fully nilpotent locus, which satisfy some stronger divisibility conditions. 

\end{remark}

\medskip

\printbibliography

\end{document}